\documentclass[11pt,oneside]{article}
\usepackage{amsfonts}
\usepackage{mathrsfs}
\usepackage{amsfonts}
\usepackage{latexsym,amssymb,amsmath}
\usepackage{latexsym}
\usepackage{amssymb}
\usepackage{amsmath}
\usepackage{amsthm}
\usepackage{indentfirst}
\usepackage{color}
\usepackage{graphicx}
\numberwithin{equation}{section} \allowdisplaybreaks

\makeatletter
   \renewcommand\@cite[1]{#1\hspace{0.2em}}
   \makeatother

\newtheorem{theorem}{\color{black}\indent Theorem}[section]
\newtheorem{lemma}{\color{black}\indent Lemma}[section]

\newtheorem{definition}{\color{black}\indent Definition}[section]
\newtheorem{remark}{\color{black}\indent Remark}[section]
\newtheorem{corollary}{\color{black}\indent Corollary}[section]
\newtheorem{example}{\color{black}\indent Example}[section]
\usepackage{amssymb,amsmath}
\usepackage{esint}
\DeclareMathOperator{\dive}{div}
\usepackage{bm}
\begin{document}
	\title{\LARGE\bf  The dependence of local regularity of solutions \\on the summability of coefficients and\\ nonhomogenous term
		\thanks{
			The research is supported by NSFC (11301211) and NSF of Jilin Prov. (201500520056JH).}\\
		\author{Zheng Li, ~Bin Guo~\thanks{Corresponding author\newline \hspace*{6mm}{\it Email
address: }bguo@jlu.edu.cn (Bin Guo)}
\\
School of Mathematics, Jilin University, Changchun 130012, PR
China}
	}
	\date{} \maketitle

{\bf Abstract:}
In this paper,  we mainly discuss the local regularity of the solution to the following problem
\begin{align*}
\begin{cases}
-\dive({\bf{A}}(x)\nabla u(x))=f(x),&~x\in\Omega,\\
u(x)=0,&~x\in\partial\Omega,
\end{cases}
\end{align*}
where $\Omega$ is a bounded domain in $\mathbb{R}^{n}$. In particular, we are concerned with the connection between the regularity of the solution $u$  and the integrability of the coefficient matrix ${\bf{A}}(x)$ as well as the nonhomogeneous term $f$. To be more precise, our first result is to prove that the maximum norm of $u$ can be controlled by $\|f\|_{s}$ with  $f\in L^s(\Omega),~s>\frac{nq}{2q-n},~q>\frac{n}{2}$.
Meanwhile, we construct some counterexamples to illustrate  the index $\frac{nq}{2q-n}$ being sharp. 
Subsequently, we give an improved upper bound for the maximum norm of $u$. Namely, there exists a positive constant $C$ such that
$$\|u\|_{\infty}\leq C\|f\|_{\frac{nq}{2q-n}}\left[\log\left(\frac{\|f\|_{s}}{~~~~~\|f\|_{\frac{nq}{2q-n}}}+1\right)+1\right].$$ Specially, the main difference of our approach compared to the arguments of [\ref{CUR}, \ref{XU}]  is to construct two classes of truncation functions  to remove the assumption of the boundedness of $u$.  
Finally, based on the previous results and Moser iteration argument, we derive the Harnack inequality of $u$ from which the H\"older continuity of the solution follows. In addition, 
we also find that the Lebesgue space $L^{\frac{n}{2}}(\Omega)$  to which  the inverse of the smallest eigenvalue $\lambda(x) $ of the matrix {\bf{A}}(x) belongs is essentially sharp 
in order to establish local boundedness and the  H\"older continuity of the solution.

\section{Introduction}
%\noindent
In this article, we consider the following second order  elliptic equation in divergence form
\begin{equation}\label{Equ01}
-\dive ({\bf{A}}(x)\nabla u) = f,~~x\in\Omega,
\end{equation}
where $\Omega$ is a bounded domain in $\mathbb{R}^n$, $n\geq2$.
%where $a:~\Omega\rightarrow \mathbb{R}^{n\times n}$~is a measurable matrix field on a domain $\Omega\subset \mathbb{R}^n,~n\geq2$.
%In order to measure ellipticity of $a$, we introduce
In order to measure ellipticity of ${\bf{A}}(x)$, the coefficient matrix ${\bf{A}}(x)$ satisfies the following assumptions:

({\rm H1}) ${\bf{A}}(x) :~\Omega\rightarrow \mathbb{R}^{n\times n}$ is a measurable matrix whose entries $a_{ij}(x)$ satisfy
$$a_{ij}(x)=a_{ji}(x),~~~ 1\leq i,j\leq n;$$

({\rm H2}) There exist $\lambda(x)$ and $\mu(x)$   such that
$$\lambda(x)|\xi|^2\leq {\bf{A}}(x)\cdot \xi\cdot\xi\leq \mu(x)|\xi|^2,~~~ \forall \xi\in \mathbb{R}^n,$$
where $\lambda(x)$ and $\mu(x)$ are nonnegative measurable functions and satisfy
$$\lambda^{-1}(x)\in L^q(\Omega),~~\mu(x)\in L^p(\Omega),$$
with $p,~q>1$ and $\frac{1}{p}+\frac{1}{q}<\frac{2}{n-1}$.
%\begin{align}\label{Cond100-1}
%\begin{cases}
%a_{ij}(x)=a_{ji}(x),~& 1\leq i,j\leq n,\\
%\lambda(x)|\xi|^2\leq a(x)\cdot \xi\cdot\xi\leq \mu(x)|\xi|^2,~& \forall \xi\in \mathbb{R}^n,\\
%\end{cases}
%\end{align}
%and suppose that $\lambda$ and $\mu$ are measurable nonnegative functions and satisfy\\
%\begin{align}\label{Cond100-2}
%\lambda^{-1}(x)\in L^q(\Omega),~\mu(x)\in L^p(\Omega),
%\end{align}\\
%where

%
When the equation \eqref{Equ01} is uniformly and strictly elliptic, i.e., $\lambda^{-1}(x)$ and $\mu(x)$ are essentially bounded,  it is well known that weak solutions are H\"{o}lder continuous. Roughly speaking,
  for the planar case ($n=2$), the corresponding study dates back to the work of Morrey [\ref{MCB}].
For more works on the best H\"older continuity exponent, the interested readers may refer to the work of Wildman  [\ref{WK}] and  Piccinini and Spagnolo [\ref{PS}]. 
For the higher dimensions ($n\geq 3$), H\"older continuity of solutions was established in the late 1950's by De Giorgi [\ref{DG}] and Nash [\ref{NJ}].  
H\"older continuity also follows from the Harnack inequality, as demonstrated by Moser [\ref{MJ}, \ref{MJ1}]. 
In particular, Trudinger [\ref{TRU2}] obtained that 
for any $f\in L^{s}(\Omega)$ with $s>\frac{n}{2}$, there exists  a constant $C>0$ such that 
\begin{align}\label{unie}
\|u\|_{{\infty}}\leqslant C\|f\|_{{s}},
\end{align}
for any nonnegative weak solution $u\in H^{1}_0(\Omega)$ of \eqref{Equ01}.
This proof relies on Moser iteration and the Sobolev inequality.
Different from the method used in  [\ref{TRU2}], Talenti [\ref{TG}] applied the rearrangement method to obtain the same results and gave the explicit expression of the constant ``$C$''. After these pioneer works, there are many research activities regarding global regularity theory for linear or nonlinear elliptic equations [\ref{GT}, \ref{LU}, \ref{MA1}, \ref{MA2}, \ref{SG}, \ref{TRU2}] and references therein.

On the contrary, if $\lambda^{-1}(x)$ is unbounded, then the equation \eqref{Equ01} is degenerate;
If $\mu(x)$ is unbounded, then  the equation \eqref{Equ01} is singular. 
Specially, Murthy and Stampacchia [\ref{MS}], Trudinger [\ref{TRU1}] proved that weak solutions to \eqref{Equ01} are locally bounded and satisfy the Harnack inequality under the assumptions that $\lambda^{-1}(x)\in L^{q}(\Omega)$ and $\mu(x)\in L^{p}(\Omega)$ with $\frac{1}{p}+\frac{1}{q}<\frac{2}{n}$. It's worth noting that their key point is to obtain the following Sobolev inequality
\begin{align}
(\int_{\Omega}|u|^{\gamma}\mu(x){\rm d}x)^{\frac{1}{\gamma}}\leqslant C\big(\int_{\Omega}|\nabla u|^{2}\lambda(x){\rm d}x\big)^{\frac{1}{2}},
\end{align}
for any $u\in H_{0}^{1}(\Omega)$ and $\gamma>2$.
Furthermore, Bella and Sch\"{a}ffner [\ref{BPMS3}] also obtained the same results under the optimal condition $\frac{1}{p} +\frac{1}{q}<\frac{2}{n-1}, n\geqslant2 $
(refer to [\ref{BPMS2}, \ref{CGMM}, \ref{HS}] for a recent generalization to  scalar autonomous integral functionals with $\left(p,q\right)$-growth). Compared to the arguments given in [\ref{MS}, \ref{TRU1}], the key observation of 
[\ref{BPMS3}] is to give the optimal Caccioppoli inequality with respect to cut-off function by using the Sobolev inequality in the $n-1$-dimensional spheres rather than  the $n$-dimensional balls.  
In addition, Fabes, Kenig and Serapioni [\ref{FKS}] gave some stronger conditions on the minimum and maximum eigenvalues of $\textbf{A}(x)$, i.e., Muckenhoupt class $A_{2}$.
These conditions will guarantee the validity of the following scale invariant Sobolev  and Poincar\'{e} inequalities like
\begin{align}\label{Ine102}
\big(\int_{B}|u-u_{B}|^{\gamma}\lambda {\rm d}x\big)^{\frac{2}{\gamma}}\leqslant Cr^{2}\big(\int_{B}\lambda^{-1} {\rm d}x\big)^{\frac{2}{\gamma}-1}\int_{B}|\nabla u|^{2}\lambda {\rm d}x,
\end{align}
for any $u\in H^{1}(B)$ with $B=B(y,r)$ and $\gamma>2$, $u_B:=\fint_{B} u {\rm d}x=\frac{1}{|B|}\int_{B} u {\rm d}x$. For more related results, we may refer to [\ref{CSWR}, \ref{CUR}, \ref{FBG}].

Recently, for the Dirichlet problem of \eqref{Equ01}, Xu [\ref{XU}] improved \eqref{unie} in the uniformly elliptic case.
Using Moser iteration technique, he obtained an upper bound in the logarithmic form of $\|f\|_{s}$. Subsequently, Cruz-Uribe and Rodney [\ref{CUR}] extended Xu's result to the degenerate elliptic equations with $f\in L^A(\Omega)$, where $A(t)=t^{\sigma^\prime}\log\left(e+t\right)^q,~\left(q>\sigma^\prime>0\right)$.

Motivated by mentioned works, we plan to investigate the relationship between the local regularity of solutions and the integrability of the minimum eigenvalue and the maximum eigenvalue of the coefficient matrix {\bf{A}}(x) as well as $f$. To the best of our knowledge, 
it is widely recognized that the smoothness of the solutions depends on the smoothness of the coefficients matrix and the data. For instance, for the absence of the term $f$, Zhong [\ref{ZX}] constructed discontinuous solutions for degenerate equations to give partial answer to De Giorgi's conjectures, which  were raised by  De Giorgi [\ref{DG1}] in a talk in Lecce, 1995.  In this direction, as far as we know, the classical results are established by Trudinger [\ref{TRU1}] under the assumption that $\frac{1}{p}+\frac{1}{q}<\frac{2}{n}$. Later,  Bella and Sch\"{a}ffner [\ref{BPMS3}] discussed the same problem under the optimal condition  $\frac{1}{p}+\frac{1}{q}<\frac{2}{n-1}$. However, for the present of the term $f$, such problem is seldom discussed. Naturally, some problems arise:

$\bullet$  Whether does the inequality similar as \eqref{unie} hold for nonuniform elliptic problem?

$\bullet$ What is a critical space of the nonhomogeneous term $f$ in the Lebesgue class $L^{s}(\Omega)$?

$\bullet$ Whether or not is the Lebesgue space $L^{\frac{n}{2}}(\Omega)$ to which  $\lambda^{-1}(x)$ belongs is sharp  to ensure the  H\"older continuity of the solution?

In this paper, we gave full answers to three problems mentioned  above. 
Here, we are interested in situations beyond any smoothness where  $a_{ij}(x)$ are assumed to be merely measurable functions satisfying conditions ({{H1}}), ({{H2}}), 
and $f$ is any function from some Lebesgue space $L^s$.
Our first finding is as follows: when $q>\frac{n}{2}$, $n\geq 2$, we prove that the maximum of the weak solution $u$ of \eqref{Equ01} can be controlled by $\|f\|_{s}$, where $s>\frac{nq}{2q-n}$.
Meanwhile, we can provide a counterexample to illustrate that $\frac{nq}{2q-n}$ is a critical exponent.
More precisely, when $q=\frac{n}{2},~n\geq 2$, we observe that for any $s\in[1,+\infty)$, there exists $f\in L^s(\Omega)$ such that the weak solution $u$ of \eqref{Equ01} is unbounded. In this case, for $n\geq3$, even when $f\in L^\infty(\Omega)$, there exists the solution satisfying $\sup\limits_{\Omega}u=+\infty$.
In addition, for the case of $q\in\left(\frac{n-1}{2},\frac{n}{2}\right),~n\geq 3$, we also find there exists $f\in L^\infty(\Omega)$ such that $\sup\limits_{\Omega} u=+\infty$. On the contrary, when $\frac{1}{p}+\frac{1}{q}\geq \frac{2}{n-1},~n\geq 3$, a counterexample is provided in [\ref{BPMS1}, \ref{FSS}], demonstrating that when $f=0$, the weak solution $u$ of \eqref{Equ01} satisfies $\sup\limits_\Omega u=+\infty$.
Finally, in the scenario where $n=2,~\frac{1}{q}+\frac{1}{p}=\frac{2}{n-1}$, i.e., $p=q=1$, Bella and Sch\"affner [\ref{BPMS3}] elucidated, using the maximum principle and Sobolev inequality in one dimension, that the weak solution $u$ is locally bounded when $f=0$.
However, we demonstrate that when $f\not\equiv 0$, for any $ s \in[1,+\infty)$, there is a counterexample which shows that \eqref{unie} is impossible.

Next we will state our main results. The first result provides an estimate of the local maximum of $|u|$.
\begin{theorem}\label{THM101}
Assume that (H1)-(H2) hold, and $f\in L^s(\Omega)$ for any $s>\frac{nq}{2q-n}~(q>\frac{n}{2})$.  Then for any $\gamma>0$, there exists $C= C(n,p,q,\gamma)\in [1,+\infty)$ such that for any ball $B_R\Subset\Omega, ~R>0$, the solution  $u$  of \eqref{Equ01} satisfies
	\begin{align}\label{th1}
	\sup_{B_{\theta R}}|u|&\leq C\left(\Lambda(B_R)\right)^{\frac{\delta p^\prime}{\gamma(\delta-1)}}\nonumber\\
	&\quad \times \left(\frac{1}{\left((1-\theta)R\right)^{\frac{n}{\gamma}(m_*+1)}}\|u\|_{L^\gamma(B_R)}+\frac{1}{\left((1-\theta)R\right)^{\frac{n}{\gamma}m_*+\frac{n}{s}-2}}\|f\|_{L^s(B_R)}\right),
	\end{align}
	where
	\begin{align*}
	\theta \in (0,1),~&~~\delta=2-\frac{1}{\chi},~~~
	\chi:=\frac{2q}{(q+1)p^*},~~~\frac{1}{p^*}=\min\left\lbrace \frac{1}{2}+\frac{1}{n-1}-\frac{1}{2p},1\right\rbrace,\\
	&m_*=\frac{p^\prime\delta}{\delta-1}\max\left\{\frac{1}{p}+\frac{1}{q}, \frac{2}{q}\right\},~~~p^\prime=\frac{p}{p-1},\\
	\Lambda(B_R)&:=\left(\fint_{B_R}\lambda^{-q}{\rm d}x\right)^\frac{1}{q}\left(\fint_{B_R}\mu^p{\rm d}x\right)^\frac{1}{p}+\left(\fint_{B_R}\lambda^{-q}{\rm d}x\right)^\frac{2}{q}.
	\end{align*}
\end{theorem}

\begin{remark}
	\rm We note that  $\frac{nq}{2q-n}\rightarrow \frac{n}{2}$ as $q\rightarrow +\infty$, which is consistent with the conclusion of the uniformly elliptic equations.
\end{remark}
\begin{remark}
\rm As a byproduct, the corresponding global results would be proceeding as in the proof of  [\ref{TRU2}, Theorem 4.1]. 
\end{remark}
\begin{remark}
	\rm From the conclusion of Theorem \ref{THM101},  we found that the integrability of $f$ depends not only on dimension  $n$, but also on the integrability of $\lambda^{-1}$. So, a natural question arises:
	  ``{\bf What is a critical exponent of the nonhomogeneous term  in the class $L^{s}(\Omega)$?}"
\end{remark}
Before  giving some answers to this problem, we first consider the special case  ${\bf{A}}(x)={\bf{I}}$. Under such case,  \eqref{Equ01} reduces the Poisson equation. We first extend $f$ to $\tilde{f}$, where
\begin{align*}
\tilde{f}=
\begin{cases}
f, & x \in B_{1}=\{x=(x_{1},x_{2},\cdots,x_{n})\in \mathbb{R}^n;~~ |x|=\sqrt{x_{1}^{2}+x_{2}^{2}+\cdots+x_{n}^{2}}<1\},\\
0, & x \in \mathbb{R}^n\setminus B_{1}.
\end{cases}
\end{align*}
Then, formally, the solution may be written as the following
\begin{align*}
u(x)=C(n)\int_{R^{n}} \tilde{f}(y)|x-y|^{2-n}{\rm d}y,~x\in B_1.
\end{align*}
 A simple computation shows the singular integral
\begin{align*}
\int_{B_1}\left(\frac{1}{|x-y|}\right)^{\frac{s(n-2)}{s-1}} {\rm d} y,~x\in B_{1}
\end{align*}
converges if and only if $s>\frac{n}{2}$. Further,  the H\"older's inequality implies 
\begin{align*}
u(x)\leq C(n)\|f\|_{L^s(B_1)}\left(\int_{B_1}\left(\frac{1}{|x-y|}\right)^{\frac{s(n-2)}{s-1}} {\rm d} y\right)^{1-\frac{1}{s}}.
\end{align*}
The formal derivation sheds light on the fact that $s>\frac{n}{2}$ is the sufficient condition for the boundedness of $|u|$. Of course, we also obtain similar analytic conclusion from scaling viewpoint. Indeed, 
 we set $u_\lambda(x)=u(\lambda x)$ and $f_\lambda(x)=\lambda^{2}f(\lambda x)$ for any $0<\lambda \ll 1$.
It is easy to verify that $\Delta u_\lambda=f_\lambda(x)$ and $\|f_\lambda(x)\|_{L^s(B_1)}=\lambda^{2-\frac{n}{s}} \|f\|_{L^s(B_{\lambda })}$.
It is also easy to verify that the necessary condition for \eqref{unie} is $s>\frac{n}{2}$ by rescaling technique.
Therefore, it is well known that $s>\frac{n}{2}$ is the sufficient and necessary condition for \eqref{unie}.
The argument above is similar in spirit, though not in detail, to the examples in [\ref{CA1}, \ref{TEVUJ}].
This leads to a natural question: is the exponent $\frac{nq}{2q-n}$ in Theorem \ref{THM101} optimal?

Next, we will demonstrate that the  exponent $\frac{nq}{2q-n}$  in Theorem \ref{THM101} is optimal. What's more, we will provide some counterexamples to illustrate the corresponding  results.
\begin{example}\label{e1}
	\rm Let $\Omega=B_\frac{1}{4}(0)$, with $q>\frac{n}{2}$ and $p,~q$ satisfy $\frac{1}{q}+\frac{1}{p}<\frac{2}{n-1}$. Here ${\bf{A}}(x)$ is a positive semidefinite diagonal matrix, and $\lambda(x)=\mu(x)=|x|^\beta\left(\log\frac{1}{|x|}\right)^\theta$, where $\beta=\frac{n}{q}<2$ and $\theta=\frac{1}{q}+\frac{1}{2}-\frac{1}{n}$.
	Then there exists $f\in L^\frac{nq}{2q-n}(\Omega)$ with $n\geq 3$, or $f\in L^s(\Omega),~s\in [1, \frac{q}{q-1})$ with $n=2$ such that the  solution $u$ of \eqref{Equ01} is unbounded.  To be precise, the functions $f(x)$ and $u(x)$ may be expressed as follows
\begin{align*}\small
f(x)&=(n+1)|x|^{\beta-1}\left(\log\frac{1}{|x|}\right)^\theta\int_{B_\frac{1}{2}(0)}|y|^{1-n}\left(|x|+|y|\right)^{-2}\frac{1}{\log\frac{1}{|y|}}{\rm d}y\\
&\quad-\theta|x|^{\beta-1}\left(\log\frac{1}{|x|}\right)^{\theta-1}\int_{B_\frac{1}{2}(0)}|y|^{1-n}\left(|x|+|y|\right)^{-2}\frac{1}{\log\frac{1}{|y|}}{\rm d}y\\
&\quad-2|x|^\beta\left(\log\frac{1}{|x|}\right)^\theta\int_{B_\frac{1}{2}(0)} |y|^{1-n}\left(|x|+|y|\right)^{-3}\frac{1}{\log\frac{1}{|y|}}{\rm d}y,
\end{align*}
and
 \begin{align*}\small
u(x)=\int_{B_\frac{1}{2}(0)}|y|^{1-n}\left(|x|+|y|\right)^{-1}\frac{1}{\log\frac{1}{|y|}}{\rm d}y.
\end{align*}
 
\end{example}
\begin{remark}
	\rm When $n=2$ and $q>1$, we can only say that $s>\frac{q}{q-1}$ is almost sharp. We conjecture that the case of $n=2$ and $f\in L^\frac{q}{q-1}\left(\Omega \right) $ is different. However, at present, we can make no further assertions regarding this.
\end{remark}
\begin{remark}
	\rm Example \ref{e1} indicates that if the assumption on $\lambda^{-1}$
	remains unchanged and the integrability of $f$ becomes weaker, then it will result in $u\notin L^\infty(\Omega)$.
\end{remark}
In Theorem \ref{THM101},  we have $\frac{nq}{2q-n}\rightarrow +\infty$ as $q\rightarrow \frac{n}{2}$.
The question arises whether $\|u\|_{\infty}$ can be controlled by $\|f\|_{\infty}$ when $q=\frac{n}{2},$ and $\frac{1}{p}+\frac{1}{q}<\frac{2}{n-1}$ for $n\geq 2$.
By the way, in this scenario, with $q^*=\frac{2nq}{(q+1)n-2q}=2$, the interpolation inequality in Theorem \ref{THM101} becomes invalid. However, we will provide the following counterexample to demonstrate when $n\geq 3$, $q=\frac{n}{2}$ and $\frac{1}{p}+\frac{1}{q}<\frac{2}{n-1}$, there exists $f\in L^\infty(\Omega)$ such that $u$ is unbounded, and when $n=2$, for any $s\in [1,+\infty)$, we can find $f\in L^s(\Omega)$ such that the maximum norm of $u$ cannot be dominated by $\|f\|_{s}$.
\begin{example}\label{e2}
	\rm Let $\Omega=B_\frac{1}{4}(0)$, with $q=\frac{n}{2}$ and $p,~q$ satisfy $\frac{1}{p}+\frac{1}{q}<\frac{2}{n-1}$. Here, ${\bf{A}}(x)$ is a positive semidefinite diagonal matrix, and $\lambda(x)=\mu(x)=|x|^2\left(\log\frac{1}{|x|}\right)^\theta$, where $\theta=\frac{5}{2n}$. Then  for $n\geq 3$, there exists $f\in L^\infty(\Omega)$, and for $n=2$, there exists $f\in L^s(\Omega)~(1 \le s<+\infty)$ such that the weak solution $u\in H^1(\Omega, {\bf{A}})$ of \eqref{Equ01} is unbounded.
\end{example}
Under our restriction $\frac{1}{p}+\frac{1}{q}<\frac{2}{n-1}$, we observe that  $q>\frac{n-1}{2}$. In the preceding discussion, we focus on the case where $q\geq \frac{n}{2}$. Now we will demonstrate that for $q\in\left(\frac{n-1}{2},~\frac{n}{2}\right)$, $n\geq 3$, there exists $f\in L^\infty(\Omega)$ such that the weak solution $u$ of \eqref{Equ01} is unbounded.
\begin{example}\label{e3}
	\rm For $n\geq 3,~q\in\left(\frac{n-1}{2},\frac{n}{2}\right)$ and $p,~q$ satisfy $\frac{1}{p}+\frac{1}{q}<\frac{2}{n-1}$. Consider the case where $\Omega=B_\frac{1}{4}(0)$, ${\bf{A}}(x)$ is a positive semidefinite diagonal matrix, and $\lambda(x)=\mu(x)=|x|^\beta\left(\log\frac{1}{|x|}\right)^\theta$, with $\beta=\frac{n}{q}>2,$ and $\theta>\frac{1}{q}$. Then there exists $f\in L^\infty(\Omega)$ such that $\sup\limits_{\Omega}u=+\infty$.
\end{example}
\begin{remark}
	\rm {Example \ref{e2} and \ref{e3} demonstrate that for $\|u\|_{\infty}<+\infty$, $L^\frac{n}{2}(\Omega)$ is the critical Lebesgue space for $\lambda^{-1}$ in divergence equations.}
\end{remark}
In addition, for the scenario where $\frac{1}{p}+\frac{1}{q}\geq\frac{2}{n-1}$ and $n\geq 3$, please refer to [\ref{BPMS3}, \ref{BPMS1}, \ref{FSS}] for related examples. 
In particular, in [\ref{FSS}, Theorem 2], an example is presented where the weak solution $u$ of \eqref{Equ01} is unbounded under the conditions $\frac{1}{p}+\frac{1}{q}>\frac{2}{n-1}$ and $f=0$
(A similar description was also provided in [\ref{BPMS3}, Remark 3.5]).
Subsequently, Bella and Sch\"affner [\ref{BPMS1}] extended the example in [\ref{FSS}] to the P-Laplacian equation when $\frac{1}{p}+\frac{1}{q}\geq \frac{m}{n-1}$ and $n\geq 3$.
It is noteworthy that they demonstrated in [\ref{BPMS3}] that when $n=2,~p=q=1,~f=0$, there exists $\|u\|_{\infty}<+\infty$.
Expanding on these findings, we establish that when $f\not\equiv 0$, $q=\frac{2}{n-1}~(n\geq 3)$ and $p=+\infty$, there exists $f\in L^\infty(\Omega)$ such that the solution $u$ of \eqref{Equ01} is unbounded. See Remark \ref{bor} below for further details.
\begin{remark}
	\rm When $n=2$, $q=1$, $p\geq1$, $f\not\equiv 0$, we are unable to establish any relationship between $\|u\|_{\infty}$ and $\|f\|_{\infty}$.
	This is because in this case, the maximum principle and the Moser iteration technique are not applicable. It may be necessary to explore a new approach or consider a counterexample in order to address this issue.
\end{remark}
Our second result is to give an improved upper bound, which may measure how much bigger the associated space $L^s(\Omega)$ is than the optimal space $L^\frac{nq}{2q-n}(\Omega)$.
\begin{theorem}\label{THM102}
	Under the assumptions in Theorem \ref{THM101}, the solution $u\in H^1_0(\Omega,{\bf{A}})$ of problem 
\eqref{Equ01} satisfies
	\begin{align}
	\|u\|_{\infty}\leq C\|f\|_{{s_0}}\left[\log\left(\frac{\|f\|_{s}}{\|f\|_{s_0}}+1\right)+1\right ],\quad
	\end{align}
	where $s_0:=\frac{nq}{2q-n}$, $C=C\left(\left\|\lambda^{-1}\right\|_q,\Omega,n,q\right)\in[1,+\infty)$. 
\end{theorem}
\begin{remark}
	\rm Compared with [\ref{CUR}, \ref{XU}], our proof also works without any assumption of the boundedness of $u$. Our main strategy is to give more precise Caccioppoli inequality for $e^{\alpha u_+}$ by applying truncation  approaching techniques.

\end{remark}
%Theorem \ref{THM102} generalizes the main result in Xu [\ref{XU}]. We note that Xu, in the uniformly elliptic case, assumed $\|u\|_{\infty}<+\infty$, and got the $\log\|f\|_s$-type upper bound of $u$ by selecting an appropriate test function $\omega^\beta-1$, where $\omega=\left(e^{u_+}-1\right)_++1$ and using the Moser iteration technique.
%We further remove the assumption that $\|u\|_{\infty}<+\infty$ and construct a new Caccioppoli inequality for $e^{\alpha u_+}$ by using two uniformly truncation functions to use Moser iteration.

Thirdly, with the aid of Theorem \ref{THM101}, we establish the Harnack inequality.
\begin{theorem}\label{THM103}
	Assuming that (H1)-(H2) hold, $q>\frac{n}{2}$ and $f\in L^s(\Omega)$ for some $s>\frac{nq}{2q-n}$.
	Let $u\in H^{1}\left(\Omega,{\bf{A}}\right)$ be a nonnegative solution in $\Omega$. Then for any $B_R\subset \Omega$, the following inequality holds
	\begin{align}\label{thm103}
	\sup_{B_\frac{R}{2}}u\leq C \left(\inf_{B_\frac{R}{2}}u+R^{2-\frac{n}{s}}\|f\|_{L^s\left(B_R\right)}\right),
	\end{align}
	where $C$ depends on $n,~p,~q$ and $\Lambda(B_R)$.
\end{theorem}
Finally, we present some consequences of Theorem \ref{THM103}, which are now considered standard and therefore we only provide the statements without the proof.
In a strictly elliptic setting where $\Lambda(B_R)$ is uniformly bounded, the Harnack inequality implies global H\"older continuity.
Due to the dependence of the constant in \eqref{thm103} on $\Lambda(B_R)$, the global H\"older continuity is generally not true in the nonuniform case (see [\ref{BPMS3}, \ref{TRU1}]).
However, Theorem \ref{THM103} leads to the following local result.
\begin{corollary}\label{Cor01}
	Suppose that (H1)-(H2) hold, $q>\frac{n}{2}$ and $f\in L^s(\Omega)$ for some $s>\frac{nq}{2q-n}$.
	Let $u\in H^1\left(\Omega,{\bf{A}}\right)$ be a weak solution of \eqref{Equ01}. Then for any $r\in(0,R]$, we have
	\begin{align*}
	\sup_{B_\frac{r}{2}} u- \inf_{B_\frac{r}{2}} u\leq C\left(\frac{r}{R}\right)^\alpha\left\{\left(\frac{1}{R^n}\int_{B_R}u^\gamma {\rm d}x\right)^\frac{1}{\gamma}+R^{2-\frac{n}{s}}\left\|f\right\|_{L^s\left(B_R\right)}\right\},
	\end{align*}
	where $C =C(\Lambda(B_R),~n,~p,~q)$, $\alpha=\alpha(\Lambda(B_R),~n,~p,~q~s)>0$.
\end{corollary}
\begin{remark}
	\rm Because $\Lambda\left(B_R\right)$ is not uniformly bounded for all ball $B_R\subset \Omega$, we only obtain this local oscillation result. Please refer to [\ref{TRU1}] for corresponding example.
\end{remark}
A direct consequence of Corollary \ref{Cor01} is the following local H\"older continuity of $u$.
\begin{corollary}
	Let $u\in H^1\left(\Omega,{\bf{A}}\right)$ be a weak solution of \eqref{Equ01}. Suppose $f\in L^s\left(\Omega\right)$ for some $s>\frac{nq}{2q-n}$. Then $u\in C^\alpha_{loc}(\Omega)$ for $\alpha\in (0,1)$ depending on $\Lambda(B_R),~n,~p,~q,~s$. Moreover, for any $B_R\subset \Omega$, the following inequality holds
	\begin{align*}
	|u(x)-u(y)|\leq C\left(\frac{|x-y|}{R}\right)^\alpha\left\{\left(\frac{1}{R^n}\int_{B_R}u^\gamma {\rm d}x\right)^\frac{1}{\gamma}+R^{2-\frac{n}{s}}\|f\|_{L^s(B_R)}\right\},
	\end{align*}
	for any $x,~y \in B_\frac{R}{8}$, where $C=C(\Lambda(B_R),n,p,q)\geq 1$.
\end{corollary}

Our main results are established through the Moser iteration.
First, Theorem \ref{THM101} demonstrates the weak solutions to nonuniformly elliptic equations are bounded with optimal index to $f$. And we explain the relation between $\|u\|_{\infty}$ and the integrability of $\lambda^{-1}$, $\mu$ as well as $f$ by Example \ref{e1}-\ref{e3}.
Next, in Theorem \ref{THM102}, we obtain critical upper bound $\|f\|_{{s_0}}\log\left(\frac{\|f\|_{s}}{\|f\|_{{s_0}}}+1\right)$ for the weak solution $u$ of the Dirichlet problem, which is smaller than the traditional upper bound $\|f\|_{s}$. And then the ratio $\frac{\|f\|_{s}}{\|f\|_{{s_0}}}$ measures how much bigger the space $L^{s}(\Omega)$ is than the critical space norm $L^{s_0}(\Omega).$
Finally, the Harnack inequality is established in Theorem \ref{THM103}.

The rest of this paper is organized as follows. In Section 2, we provide some preliminary knowledge, including some definitions and useful lemmas.
In Section 3, we prove the Theorem \ref{THM101}, by utilizing a variation of the Moser iteration when $\frac{1}{p}+\frac{1}{q}<\frac{2}{n-1}$, and carefully modifying the argument from [\ref{BPMS3}].
Section 4 is dedication to proving Theorem \ref{THM102}.
Our proof is a generalization of the argument in [\ref{XU}], and it requires us to address several technical obstacles.
Section 5 contains the proof of Theorem \ref{THM103}.
Finally, in Section 6, we establish Example \ref{e1}.
In a sense, Example \ref{e2}, \ref{e3} can be considered as variants of Example \ref{e1}.

\section{Preliminaries}
In this section, we introduce some notations and lemmas that will be used throughout the paper. In what follows,  we denote by $\|\cdot\|_s~(s\ge 1)$ the usual norm in $L^s(\Omega)$. And $C$ denotes a generic positive constant, which may differ at each appearance, but whose value
depends only on the underlying parameters. If we want to specify this dependence, we will write,
for instance, $C(n,s)$, etc.

The spaces $H^{1}_0(\Omega,{\bf{A}})$ and $H^1(\Omega,{\bf{A}})$ are,  respectively, defined as the completion of $C^\infty_0(\Omega)$ and $C^\infty(\Omega)$ with respect to the norm $$\|\cdot\|_{H^1(\Omega,{\bf{A}})}:=\left(\mathcal{B}(\cdot,\cdot)\right)^\frac{1}{2},$$
	where
	$$\mathcal{B}(u,v):=\int_{\Omega}{\bf{A}}\cdot\nabla u\cdot\nabla v {\rm d}x+\int_{\Omega}\mu u v {\rm d}x,$$

For the related properties of the spaces $H^1(\Omega,{\bf{A}})$ and $H^1_0(\Omega,{\bf{A}})$, please refer to [\ref{TRU1}, \ref{TRU2}]. Here, we will only recall the following chain rule.
\begin{lemma}$^{[\ref{TRU2}]}$\label{lian}
	{\rm Let $F:\mathbb{R}\rightarrow \mathbb{R}$ be uniformly Lipschitz-continuous with $F(0)=0$.  Then $u\in H^1_0(\Omega,{\bf{A}})$ (or $\in H^1(\Omega,{\bf{A}}))$ implies $F(u)\in H^1_0(\Omega,{\bf{A}})$ (or $\in H^1(\Omega,{\bf{A}})$), and it holds that $\nabla_x F=F^\prime(u)\nabla_x u$ a.e. $\Omega$}.
\end{lemma}

Next, we introduce the definition of weak solution (subsolution,  supersolution) of \eqref{Equ01}.
%Before we state the proof of the results, we first define the notion of weak solution to \eqref{Equ01} that we consider here.
\begin{definition}
	The function $u$ is said to be a weak solution (subsolution,  supersolution) of \eqref{Equ01} in $\Omega$ if and only if $u\in H^1(\Omega,{\bf{A}})$ and
	\begin{align}\label{def1}
\mathcal{B}(u,\phi)=\int_{\Omega}f \phi {\rm d}x ~(\leq 0,~\geq 0),  	~~\forall \phi \ge 0\in H^1_0(\Omega,{\bf{A}}).
	\end{align}
	Moreover, $u$ is said to be a local weak solution of \eqref{Equ01} in $\Omega$ if and only if $u$ is a weak solution of \eqref{Equ01} in $\Omega^\prime$ for any bounded open set $\Omega^\prime\Subset \Omega$.
\end{definition}
\begin{remark}
	{\rm Since $f\in L^s(\Omega),~s>\frac{nq}{2q-n}$, the right-hand side  of \eqref{def1} is valid according to H\"older's inequality and Sobolev inequality.}
\end{remark}

At the end of this section, we introduce another important lemma, which was first proposed in [\ref{BPMS3}], and provided improved Caccioppoli inequality by selecting the minimum points of the functional $J$ and applying the Sobolev inequality on spheres. For a detailed proof, please refer to [\ref{BPMS3}, Lemma 2.1].
\begin{lemma}$^{[\ref{BPMS3}]}$\label{lem2.1}
	Fix $n\geq 2$, and $p\geq 1$ satisfy $p>\frac{n-1}{2}$ if $n\geq 3$.
	For $0<\rho<\sigma<+\infty$, let $v\in W^{1,p^*}(B_\sigma)$ with $\frac{1}{p^*}=\min\{\frac{1}{2}-\frac{1}{2p}+\frac{1}{n-1},1\}$ and $\mu\in L^p(B_\sigma)$, $\mu\geq 0$ satisfying $\mu v^2\in L^1(B_\sigma)$. Consider
	\begin{align}
	J(\rho,\sigma,v):=\inf\left\{\int_{B_\sigma}\mu v^2|\nabla\eta|^2{\rm d}x~\bigg|~ \eta\in C^1_0(B_\sigma),~\eta\geq 0,~\eta\equiv1 \text{ in }B_\rho\right\}.
	\end{align}
	Then there exists $c=(n,p)\in[1,+\infty)$ such that
	\begin{align}
	J(\rho,\sigma,v)\leq c(\sigma-\rho)^{-\frac{2n}{n-1}}\|\mu\|_{L^p(B_\sigma\setminus B_\rho)}\left(\|\nabla v\|_{L^{p^*}(B_\sigma\setminus B_\rho)}^2+\rho^{-2}\|v\|_{L^{p^*}(B_\sigma\setminus B_\rho)}^2\right).
	\end{align}
\end{lemma}

\section{Proof of Theorem \ref{THM101}}

To better understand the outline of our proof, we will divide this proof into three steps.

{\bf Step 1. Establish the estimate of subsolution for $ \theta=\frac{1}{4}$, $ R=1$ and $\gamma= 2p^\prime$,~where $ p^\prime=\frac{p}{p-1}$.}
i.e., if $u$ is subsolution of \eqref{Equ01} for $\gamma= 2p^\prime$, then we claim that
\begin{align}
\|u_+\|_{L^\infty(B_\frac{1}{4})}\leq C \left(\Lambda(B_1)\right)^{\frac{\delta}{2(\delta-1)}}\left[\left(\int_{B_1}|u_+|^{2p^\prime} {\rm d}x\right)^\frac{1}{2p^\prime}+\|f\|_{L^s(B_1)}\right].
\end{align}

To do this, we first introduce the definition of some test function. For $\beta\geq 1$ and $m\in (0,+\infty)$, we define
\begin{align*}
\phi_m=\begin{cases}
(u_++k)^\beta-k^\beta,~&u\leq m,\\
\beta(m+k)^{\beta-1}(u-m)+(m+k)^\beta-k^\beta,~&u>m,
\end{cases}
\end{align*}
where $k:=\|f\|_{L^s(B_1)}$.
Set $\phi=\phi_m\eta^2$, where $\eta(x)\geq 0,~\eta\in C^\infty_0(B_1)$, with $\eta\equiv 1$ for $x\in B_\rho$; $\eta\equiv 0$ for $x\in B_1\setminus B_\sigma$, where $\frac{1}{4}\leq \rho<\sigma\leq \frac{1}{2}$.
Using \eqref{def1}, we obtain
\begin{align}\label{fev}
\int_{\Omega}{\bf{A}}(x)\nabla u\nabla &\left(\phi_m\eta^2\right) {\rm d}x \leq \int_{\Omega}f \eta^2\phi_m {\rm d}x.
\end{align}
Then, by means of \eqref{fev}, the conditions (H1)-(H2), Young's inequality, and convexity of $\phi_m$ in the form of $\phi_m\leq (u_++k)\phi^\prime_m$, we get
\begin{align}\label{qw}
\int_{\Omega}\left({\bf{A}}(x)\nabla u_+\nabla u_+\right)\phi^\prime_m\eta^2 {\rm d}x
&\leq \frac{1}{2}\int_{\Omega}{\bf{A}}(x)\nabla u_+\nabla u_+\phi^\prime_m\eta^2 {\rm d}x 
\nonumber\\
&\quad+8\int_{\Omega}\mu |\nabla \eta|^2(u_++k)^2\phi^\prime_m {\rm d}x +\int_{\Omega}f \eta^2\phi_m {\rm d}x .
\end{align}
Once again, using the elliptic condition ${\bf{A}}(x)\nabla u_+\nabla u_+\geq \lambda(x)|\nabla u_+|^2$, it follows from \eqref{qw} that
%by absorbing the first term on the right-hand side of \eqref{qw} into left-hand side gets
%\begin{align}\label{we}
%\int_{\Omega}\left(a(x)\nabla u_+\nabla u_+\right)\phi^\prime_m\eta^2 {\rm d}x
%\leq 16 \int_{\Omega}\mu(x)|\nabla \eta|^2(u_++k)^2 \phi^\prime_m {\rm d}x
%+2\int_{\Omega}f \eta^2 \phi_m {\rm d}x.
%\end{align}
%Moreover, noticing that $a(x)\nabla u_+\nabla u_+\geq \lambda(x)|\nabla u_+|^2$, it follows from \eqref{we} that
\begin{align}\label{zo}
\int_{\Omega}\lambda|\nabla u_+|^2\phi^\prime_m\eta^2 {\rm d}x
\leq 16 \int_{\Omega}\mu|\nabla \eta|^2(u_++k)^2 \phi^\prime_m {\rm d}x
+2\int_{\Omega}f \eta^2 \phi_m {\rm d}x.
\end{align}
Imposing H\"older's inequality on the last term of \eqref{zo}, we obtain
\begin{align}\label{dai1}
\int_{B_\sigma} f \phi_m \eta^2 {\rm d}x
\leq \int_{B_\sigma}\frac{|f|}{k}(u_++k)\phi_m\eta^2 {\rm d}x
  \leq  \left(\int_{B_\sigma}
\left|(u_++k)^\frac{1}{2}\phi_m^\frac{1}{2}\eta\right|^{2s^\prime}{\rm d}x\right)^\frac{1}{s^\prime},
\end{align}
where $s'=\frac{s}{s-1}.$

Next, we estimate the right-hand side of \eqref{dai1}. Actually, $q>\frac{n}{2}$ implies $p^*<2s^\prime=\frac{2s}{s-1}<q^\ast:=\frac{2nq}{(q+1)n-2q}$. And then to combine Sobolev inequality and interpolation inequality yields the following inequalities 
\begin{align}\label{dai2}
\left(\int_{B_\sigma }\left|(u_++k)^\frac{1}{2}\phi_m^\frac{1}{2}\eta\right|^{2s^\prime}{\rm d}x\right)^\frac{1}{s^\prime}\nonumber
&\leq \varepsilon \left(\int_{B_\sigma}\left|(u_++k)^\frac{1}{2}\phi_m^\frac{1}{2}\eta\right|^{q^*}{\rm d}x\right)^\frac{2}{q^*} \\
&\quad+\frac{1}{\varepsilon}\left(\int_{B_\sigma}\left|(u_++k)^\frac{1}{2}\phi_m^\frac{1}{2}\eta\right|^{p^*} {\rm d}x\right)^\frac{2}{p^*}\nonumber\\
&\leq \varepsilon\left(\int_{B_\sigma}\left|\nabla\left((u_++k)^\frac{1}{2}\phi_m^\frac{1}{2}\eta\right)\right|^\frac{2q}{q+1} {\rm d}x \right)^\frac{q+1}{q}\nonumber\\
&\quad+\frac{1}{\varepsilon}\left(\int_{B_\sigma}\left|(u_++k)^\frac{1}{2}\phi_m^\frac{1}{2}\eta\right|^{p^*} {\rm d}x\right)^\frac{2}{p^*},
\end{align}
for sufficiently small $\varepsilon$ (which will be determined later). Furthermore, keeping H\"older's inequality and Cauchy-Schwarz inequality in mind, we estimate the first term of the right-hand side of \eqref{dai2} as follows
\begin{align}\label{dai3}
\left(\int_{B_\sigma}\left|\nabla\left((u_++k)^\frac{1}{2}\phi_m^\frac{1}{2}\eta\right)\right|^{\frac{2q}{q+1}} {\rm d}x\right)^\frac{q+1}{q}&\leq \left\|\lambda^{-1}\right\|_{L^q(B_\sigma)}\left(\int_{B_\sigma}\lambda\left|\nabla \left((u_++k)^\frac{1}{2}\phi_m^\frac{1}{2}\eta\right)\right|^{2} {\rm d}x\right)\nonumber\\
&\leq \left\|\lambda^{-1}\right\|_{L^q(B_\sigma)}\left(2\int_{B_\sigma}\lambda \left|\nabla \left((u_++k)^\frac{1}{2}\phi^\frac{1}{2}_m\right)\right|^2\eta^2  {\rm d}x \right.\notag\\
&\quad\left.+\;2\int_{B_\sigma}\lambda |\nabla \eta|^2(u_++k)\phi_m {\rm d}x\right)\nonumber\\
&\leq  \left\|\lambda^{-1}\right\|_{L^q(B_\sigma)}\left(2\int_{B_\sigma}\lambda \left|\nabla \left((u_++k)^\frac{1}{2}\phi^\frac{1}{2}_m\right)\right|^2\eta^2  {\rm d}x \right.\notag\\
&\quad\left.+\;2\int_{B_\sigma}\mu |\nabla \eta|^2(u_++k)\phi_m {\rm d}x\right).
\end{align}
By applying H\"older's inequality and  substituting \eqref{dai1}, \eqref{dai2}, \eqref{dai3} into \eqref{zo}, we conclude
\begin{align}\label{rcaci}
\left(\int_{B_\rho}\left|\eta\nabla G_m(u)\right|^\frac{2q}{q+1} {\rm d}x \right)^\frac{q+1}{q}
&\leq \left\|\lambda^{-1}\right\|_{L^q(B_\sigma)}\int_{B_\sigma}\lambda|\nabla G_m(u)|^2\eta^2 {\rm d}x\nonumber\\
&\leq  C\left\|\lambda^{-1}\right\|_{L^q(B_\sigma)} \int_{B_\sigma}\mu|\nabla \eta|^2(u_++k)^2\phi^\prime_m {\rm d}x \nonumber\\ 
&\quad+C\left\|\lambda^{-1}\right\|_{L^q(B_\sigma)}^2\left(\int_{B_\sigma}\left|(u_++k)^2\phi^\prime_m\eta\right|^\frac{p^*}{2} {\rm d}x\right)^\frac{2}{p^*},
\end{align}
where $G_m(u)=\displaystyle \int_{0}^{u_+}\left|\phi^\prime_m(s)\right|^\frac{1}{2} {\rm d}
s$.

To optimize the first term on the right-hand side of  \eqref{rcaci} with respect to $\eta$, we need to verify that $\mu(u_++k)^2\phi^\prime_m$ and $(u_++k)\left(\phi^\prime_m\right)^\frac{1}{2}$ satisfy the conditions of Lemma \ref{lem2.1}. Indeed, it is not hard to find $\mu(u_++k)^2\phi^\prime_m \in L^1(B_\sigma)$ and $(u_++k)\left(\phi^\prime_m\right)^\frac{1}{2}\in W^{1,p^*}(B_\sigma)$ by virtue of $u\in H^1(\Omega,{\bf{A}})$, $p^*<\frac{2q}{q+1}$ as well as the properties of $\phi_m^\prime$. Further Lemma \ref{lem2.1} shows 
\begin{align}\label{eta}
\inf_{\eta}\left\{\int_{B_\sigma}\mu\left|\nabla \eta\right|^2(u_++k)^2\phi^\prime_m{\rm d}x\right\}
\leq C   (\sigma-\rho)^{-\frac{2n}{n-1}}\|\mu\|_{L^p(B_\sigma)}\left\|(u_++k)\left(\phi^\prime_m\right)^\frac{1}{2}\right\|^2_{W^{1,p^*}(B_\sigma)}.
\end{align}
On the other hand, when $m$ goes to infinity, we have
\begin{align*}
&\varliminf\limits_{m\rightarrow +\infty}|\nabla G_m(u)|^2= \varliminf\limits_{m\rightarrow +\infty} \left|G^\prime_m(u)\right|^2|\nabla u_+|^2 \geq\; \beta (u_++k)^{\beta-1}|\nabla u_+|^2~a.e. ,\\
&(u_++k)^2\phi^\prime_m \leq\;\beta(u_++k)^{\beta+1},~~\forall m\in [1,+\infty).
\end{align*}
Obviously, to use \eqref{rcaci}, \eqref{eta} and $\eta\equiv1$ in $B_\rho$ derives the following inequality 
\begin{align}\label{cdiedai}
\left(\int_{B_\rho}\left|\nabla \left((u_++k)^\frac{\beta+1}{2}\right)\right|^\frac{2q}{q+1}{\rm d}x\right)^\frac{q+1}{2q}&\leq C (\beta+1)(\sigma-\rho)^\frac{-n}{n-1}\left(\Lambda(B_\sigma)\right)^\frac{1}{2}\notag\\
&\quad\times\left\|(u_++k)^\frac{\beta+1}{2}\right\|_{W^{1,p^*}(B_\sigma)}.
\end{align}
To make Moser iteration argument work, we further prove that there exists $\delta>1$ such that the term $\|\left(u_++k\right)^\frac{(\beta+1)\delta}{2}\|_{W^{1,p^*}(B_\rho)}$ may be dominated by $\|(u_++k)^{\frac{\beta+1}{2}}\|_{W^{1,p^*}(B_\sigma)}$. Indeed, on the one hand, choose $\delta=2-\frac{1}{\chi}>1$ with $\chi=\frac{2q}{(q+1)p^*}>1$, and  apply H\"older's inequality with exponent $\frac{2q}{(q+1)p^*}$  to have 
\begin{align}\label{dao}
\left(\int_{B_\rho}\left|\nabla\left((u_++k)^{\frac{\beta+1}{2}\delta}\right)\right|^{p^*} {\rm d}x\right)^\frac{1}{p^*}
&\leq \left(\int_{B_\rho}\left|\delta(u_++k)^{\frac{\beta+1}{2}(\delta-1)}\nabla \left((u_++k)^\frac{\beta+1}{2}\right)\right|^{p^*}{\rm d}x\right)^\frac{1}{p^*}\nonumber\\
&\leq \delta\left(\int_{B_\rho}\left|\nabla \left((u_++k)^\frac{\beta+1}{2}\right)\right|^\frac{2q}{q+1}{\rm d}x\right)^\frac{q+1}{2q}\notag\\
&\quad\times\left(\int_{B_\rho}\left|(u_++k)^\frac{\beta+1}{2}\right|^{p^*} {\rm d}x\right)^\frac{\delta-1}{p^*}.
\end{align}
On the other hand, with the aid of Sobolev inequality $W^{1,p^*}(B_\rho)\hookrightarrow L^t(B_\rho)$ where $1\leq t\leq \frac{np^*}{n-p^*}$, we have 
\begin{align}\label{qia}
\left(\int_{B_\rho}(u_++k)^{\frac{\beta+1}{2}\delta p^*} {\rm d}x\right)^\frac{1}{\delta p^*}
\leq C \left\|(u_++k)^\frac{\beta+1}{2}\right\|_{W^{1,p^*}(B_\rho)}.
\end{align}
Thus, combining \eqref{cdiedai}, \eqref{dao} with \eqref{qia},  we have
\begin{align}\label{fg}
\left\|(u_++k)^{\frac{\beta+1}{2}\delta}\right\|_{W^{1,p^*}(B_\rho)}\leq C (\beta+1)\left(\Lambda(B_1)\right)^\frac{1}{2}(\sigma-\rho)^\frac{-n}{n-1}\left\|(u_++k)^\frac{\beta+1}{2}\right\|_{W^{1,p^*}(B_\sigma)}^\delta.
\end{align}
Set $\alpha=\frac{\beta+1}{2}\geq 1$, then  \eqref{fg} can be rewritten as
\begin{align}\label{ghj}
\left\|(u_++k)^{\alpha\delta}\right\|_{W^{1,p^*}(B_\rho)}^{\frac{1}{\delta\alpha}}\le \left(C\alpha\left( \Lambda(B_1)\right) ^\frac{1}{2}(\sigma-\rho)^{\frac{-n}{n-1}}\right)^\frac{1}{\delta\alpha}\left\|(u_++k)^\alpha\right\|_{W^{1,p^*}(B_\sigma)}^\frac{1}{\alpha}.
\end{align}
We further choose $\alpha=\delta^m$, $\rho_m=\frac{1}{4}+\left(\frac{1}{4}\right)^{m+1},  m\in\mathbb{N}$, and deduce the following recurrence relations for the function $u_{+}+k$
\begin{align}
\left\|(u_++k)^{\delta^{m+1}}\right\|_{W^{1,p^*}(B_{\rho_{m+1}})}^{\frac{1}{\delta^{m+1}}}
&\le  \left(\left( \Lambda(B_1)\right) \right)^\frac{1}{2\delta^{m+1}}\delta^{\frac{m}{\delta^{m+1}}}\left(C4^\frac{n}{n-1}\right)^{\frac{m+1}{\delta^{m+1}}}\notag\\
&\quad\times\left\|(u_++k)^{\delta^m}\right\|_{W^{1,p^*}(B_{\rho_m})}^\frac{1}{\delta^m}.
\end{align}
By iteration, we find
\begin{align}\label{infty}
\|u_++k\|_{L^\infty(B_\frac{1}{4})}
\leq C \left[\left( \Lambda(B_1)\right) \right]^{\sum_{i=0}^{\infty}\frac{1}{2\delta^{i+1}}}\left(4^{\frac{-n}{n-1}}\delta\right)^{\sum_{i=0}^{\infty}\frac{i+1}{\delta^{i+1}}}\|(u_++k)\|_{W^{1,p^*}(B_\frac{1}{2})}.
\end{align}
Subsequently, we estimate the right-hand side of \eqref{infty}.
Using \eqref{rcaci} with $\beta=1, ~\rho=\frac{1}{2}, ~\sigma=1$, the fact $p^*<\frac{2q}{q+1}\leq 2\leq 2p^\prime$, and $\|\nabla \eta\|_{L^\infty(B_\sigma)}\leq c(\sigma-\rho)^{-1}$, we obtain
\begin{align*}
\left(\int_{B_\frac{1}{2}}|\nabla (u_++k)|^{p^*} {\rm d}x\right)^\frac{1}{p^*}
&\leq \left(\int_{B_\frac{1}{2}}|\nabla (u_++k)|^\frac{2q}{q+1} {\rm d}x \right)^\frac{q+1}{2q}\nonumber\\
&\leq   \left( \Lambda(B_1)\right) ^\frac{1}{2}\left[\left(\int_{B_1} |u_++k|^{2p^\prime} {\rm d}x \right)^\frac{1}{2p^\prime}\right.\notag\\
&\quad\left.+\left(\int_{B_1}|u_++k|^{p^*} {\rm d}x \right)^\frac{1}{p^*}\right].
\end{align*}
And then by H\"older's inequality, we get
\begin{align*}
\left(\displaystyle \int_{B_1}|u_++k|^{p^*}{\rm d}x\right)^\frac{1}{p^*}\leq C \left(\displaystyle \int_{B_1}|u_++k|^{2p^\prime} {\rm d}x\right)^\frac{1}{2p^\prime}.
\end{align*}
In turn, combining the above inequalities with \eqref{infty}, we obtain
\begin{align}\label{st}
\|u_+\|_{L^\infty(B_\frac{1}{4})}&\leq C\left(\Lambda(B_1)\right)^\frac{\delta}{2(\delta-1)}\left(\int_{B_1}|u_++k|^{2p^\prime} {\rm d}x\right)^\frac{1}{2p^\prime}\nonumber\\
&\leq C\left( \Lambda(B_1)\right) ^\frac{\delta}{2(\delta-1)}\left[\left(\int_{B_1}|u_+|^{2p^\prime} {\rm d}x \right)^\frac{1}{2p^\prime}+\|f\|_{L^s(B_1)}\right],
\end{align}
which proves the claim.

{\bf Step 2. The general case.} Although the procedure is standard, we also give a sketch of our proof to make our article self-contained. For more details, we may refer to  [\ref{HL}]. First, let
$$\tilde{{\bf{A}}}(y)={\bf{A}}(Ry),~\tilde{\lambda}(y)=\lambda(Ry),~\tilde{\mu}(y)=\mu(Ry),~\tilde{u}(y)=u(Ry),~\tilde{f}(y)=R^2f(Ry),~\tilde{k}=\left\|\tilde{f}\right\|_{L^s(B_1)},$$
where $y\in B_1$.
We apply what we just proved to $\tilde{u}$ in $B_1$ and rewrite the result in terms of $u$, then for any $\gamma\geq 2p^\prime$, we derive
\begin{align}\label{re}
\|u_+\|_{L^\infty\left(B_{\frac{1}{4}R}\right)}\leq C\left(\Lambda(B_R)\right)^\frac{p^\prime\delta}{\gamma(\delta-1)}\left(R^{-\frac{n}{\gamma}}\left(\int_{B_R}|u_+|^{\gamma} {\rm d}x\right)^\frac{1}{\gamma}+R^{2-\frac{n}{s}}\|f\|_{L^s(B_R)}\right),
\end{align}
where $C=C(p,q,n,\gamma)\geq 1$.
Again, by the same argument as above, \eqref{re} clearly also holds for $B_R$ replaced by $B_{(1-\theta)R}(x_0)$, where $x_0\in B_{\theta R}$, we arrive at
\begin{align}\label{1}
\|u_+\|_{L^\infty\left(B_{\frac{1}{4}(1-\theta)R}(x_0)\right)} 
&\leq C\left[\Lambda\left(B_{(1-\theta)R}(x_0)\right)\right]^\frac{p^\prime\delta}{\gamma(\delta-1)}\left\{\left[(1-\theta)R\right]^{-\frac{n}{\gamma}}\left(\int_{B_{(1-\theta)R}(x_0)}|u_+|^{\gamma} {\rm d}x\right)^\frac{1}{\gamma}\right.\notag\\
&\quad\left.+\left[(1-\theta)R\right]^{2-\frac{n}{s}}\|f\|_{L^s(B_R)}\right\}\nonumber\\
&\leq C\left(\Lambda(B_R)\right)^\frac{p^\prime\delta}{\gamma(\delta-1)}\left[\frac{1}{(1-\theta)^{\frac{n}{\gamma}(m_*+1)}}R^{-\frac{n}{\gamma}}\left(\int_{B_R}|u_+|^{\gamma} {\rm d}x\right)^\frac{1}{\gamma}\right.\notag\\
&\quad\left.+\frac{1}{(1-\theta)^{\frac{n}{\gamma}m_*+\frac{n}{s}-2}}R^{2-\frac{n}{s}}\|f\|_{L^s(B_R)}\right],
\end{align}
where $m_*=\frac{p^\prime\delta}{\delta-1}\max\left\{\frac{1}{p}+\frac{1}{q},\frac{2}{q}\right\}$, $C=C(n,p,q,\gamma)$.

Next, we will be in position to prove \eqref{th1} for $\gamma\in (0,2p^\prime)$. First, we observe that
\begin{align}\label{wer}
\left(\int_{B_R}|u_+|^{2p^\prime} {\rm d}x\right)^\frac{1}{2p^\prime}
\leq   \|u_+\|_{L^\infty(B_R)}^{\frac{2p^\prime-\gamma}{2p^\prime}}\|u_+\|_{L^\gamma(B_R)}^\frac{\gamma}{2p^\prime}.
\end{align}
In addition, applying the Young's inequality to \eqref{wer}, and combining this with \eqref{1} for $\gamma=2p^\prime$, we obtain
\begin{align}\label{st}
\|u_+\|_{L^\infty(B_{\theta R})}&\leq  \frac{1}{2}\|u_+\|_{L^\infty(B_R)}+C\left(\Lambda(B_R)\right)^\frac{\delta p^\prime}{\gamma(\delta-1)}\frac{1}{(1-\theta)^{\frac{n}{\gamma}(m_*+1)}}R^{-\frac{n}{\gamma}}\|u_+\|_{L^\gamma(B_R)}\nonumber\\
&\quad +(\Lambda(B_R))^\frac{\delta}{2(\delta-1)}\frac{1}{(1-\theta)^{\frac{n}{\gamma}m_*+\frac{n}{s}-2}}R^{2-\frac{n}{s}}\|f\|_{L^s(B_R)},
\end{align}
where $C=C(n,p,q,\gamma)$.
Then, set $r:=\theta R~(\theta R<R\leq 1,~\theta\in (0,1))$ and apply  \eqref{st} to get
\begin{align}\label{lin}
\left\|u_+\right\|_{L^\infty(B_r)}
&\leq \frac{1}{2}\|u_+\|_{L^\infty(B_R)}+C\left( \Lambda(B_R)\right) ^\frac{\delta p^\prime}{\gamma(\delta-1)}\frac{1}{\left(1-\frac{r}{R}\right)^{\frac{n}{\gamma}(m_*+1)}}R^{-\frac{n}{\gamma}}\|u_+\|_{L^\gamma(B_1)}\nonumber\\
&\quad+C(\Lambda(B_R))^\frac{\delta}{2(\delta-1)}\frac{1}{\left(1-\frac{r}{R}\right)^{\frac{n}{\gamma}m_*+\frac{n}{s}-2}}R^{2-\frac{n}{s}}\|f\|_{L^s(B_1)}\nonumber\\
&\leq\frac{1}{2}\|u_+\|_{L^\infty(B_R)}+C\left( \Lambda(B_1)\right) ^\frac{\delta p^\prime}{\gamma(\delta-1)}\frac{1}{(R-r)^{\frac{n}{\gamma}(m_*+1)}}\|u_+\|_{L^\gamma(B_1)}\notag\\
&\quad+C\left( \Lambda(B_1)\right) ^\frac{\delta}{2(\delta-1)}\frac{1}{(R-r)^{\frac{n}{\gamma}m_*+\frac{n}{s}-2}}\|f\|_{L^s(B_1)}\nonumber\\
&\leq\frac{1}{2}\|u_+\|_{L^\infty(B_R)}
+C\left( \Lambda(B_1)\right) ^\frac{\delta p^\prime}{\gamma(\delta-1)}\notag\\
&\quad\times\left(\frac{1}{(R-r)^{\frac{n}{\gamma}(m_*+1)}}\|u_+\|_{L^\gamma(B_1)}
+\frac{1}{(R-r)^{\frac{n}{\gamma}m_*+\frac{n}{s}-2}}\|f\|_{L^s(B_1)}\right),
\end{align}
where $C=C(n,p,q,\gamma)$. Let $h(r):=\|u_+\|_{L^\infty(B_r)}$, then \eqref{lin} becomes
\begin{align*}
h(r)&\leq \frac{1}{2}h(R)+C\left( \Lambda(B_1)\right) ^\frac{\delta p^\prime}{\gamma(\delta-1)}
\left(\frac{1}{(R-r)^{\frac{n}{\gamma}(m_*+1)}}\|u_+\|_{L^\gamma(B_1)}\right.\notag\\
&\quad\left.+\frac{1}{(R-r)^{\frac{n}{\gamma}m_*+\frac{n}{s}-2}}\|f\|_{L^s(B_1)}\right).
\end{align*}
Proceeding as a proof of  Lemma 4.3  in [\ref{HL}],  we may derive 
\begin{align}
\|u_+\|_{L^\infty\left(B_\theta\right)}
&\leq C\left( \Lambda(B_1)\right) ^\frac{\delta p^\prime}{\gamma(\delta-1)} \notag\\
&\quad \times \left(\frac{1}{(1-\theta)^{\frac{n}{\gamma}(m_*+1)}}\|u_+\|_{L^\gamma(B_1)}+\frac{1}{(1-\theta)^{\frac{n}{\gamma}m_*+\frac{n}{s}-2}}\|f\|_{L^s(B_1)}\right),
\end{align}
where $C=C(n,p,q,\gamma)$.

{\bf Step 3. Consider \bm{$\|(-u)_+\|_{L^\infty\left(B_\theta\right)}$}.}
We consider $u$ is supersolution of \eqref{Equ01}, and rewrite the above result in terms of $-u$, then we could prove the claim. 

The proof of Theorem \ref{THM101} is complete.
\begin{remark}
	{\rm The proof of Theorem \ref{THM101}, clearly shows that the decrease in the integrability of $\lambda^{-1}$ leads to an increase in the integrability index of $f$ compared to uniform case. However, this does not imply that the integrability assumption of $\mu$ is redundant. Indeed, by virtue of $p^*<\frac{2q}{q+1}$,  we also obtain $p>\frac{1}{\frac{2}{n-1}-\frac{1}{q}}$. Therefore, a sufficiently strong integrability assumption for $\mu$ is indispensable for the continuation of the proof.}
\end{remark}

\section{Proof of Theorem \ref{THM102}}
In this section, without loss of generality, we can assume that $\|f\|_{\frac{nq}{2q-n}}=1$ for $f \not\equiv 0$. We shall begin with the Lemma \ref{cs}.  This lemma is mainly based on [\ref{CUR}, Lemma 2.22] and [\ref{XU}, Lemma B], but in comparison, we do not assume that $u$ is bounded. Our idea of proof is motivated by [\ref{CSL}, \ref{CA}].
\begin{lemma}\label{cs}
	Assume that (H1)-(H2) hold, $q>\frac{n}{2}$ and $f \in L^s(\Omega)$ for some $s>\frac{nq}{2q-n}.$ Let $u$ be the non-negative, weak subsolution of \eqref{Equ01} in $\Omega$ with zero Dirichlet boundary value.
	Then for any $\alpha \in \left(0, \frac{4}{c_s^2\|\lambda^{-1}\|_{q}} \right) ,$ where $c_s$ is the sharp constant of the Sobolev inequality\footnote{Here the Sobolev inequality is $\|u\|_{p*}\le c_s\|\nabla u\|_{p}$, where $p^*=\frac{np}{n-p},$ $u \in W_0^{1,p}(\Omega)$.}, there exists $C=C(c_s, \Omega, n, q)$ such that
	\begin{align*}
	\left(\int_\Omega e^\frac{\alpha q^*u}{2}{\rm d}x\right)^\frac{1}{q^*}\leq \frac{2\alpha|\Omega|^{1-\frac{(q^*)^\prime}{s_0}}}{\left(2-c_s\left\|\lambda^{-1}\right\|^\frac{1}{2}_{q}\sqrt{\alpha}\right)^2}+|\Omega|^\frac{1}{q^*},
	\end{align*}
	where $q^*=\frac{2nq}{n(q+1)-2q}$.
\end{lemma}

\begin{proof}
	First we fix $\alpha\in \left(0,\frac{4}{c_s^2\|\lambda^{-1}\|_{q}}\right)$
	and define
	\begin{align*}
	\varphi(u)=\phi^\alpha_{N}(u)=
	\begin{cases}
	e^{\alpha u}-1,&~u\leq N,\\
	\alpha e^{\alpha N}u +(1-\alpha N)e^{\alpha N}-1,&~u>N.
	\end{cases}
	\end{align*}
	Using Lemma \ref{lian}, we have $\phi^\alpha_{N}(u) \in H^1_0(\Omega, \bf{A}).$ Substituting $\phi^\alpha_{N}(u)$ into \eqref{def1} yields
	\begin{align}\label{1.2.2}
	\int_{\Omega}{\bf{A}}(x)\nabla u\nabla \phi^\alpha_N {\rm d}x
	&=\int_\Omega {\bf{A}}(x)\nabla u\nabla u \left(\phi _N^\alpha \right)^\prime {\rm d}x\nonumber\\
	&\leq \int_\Omega |f \phi_{N}^{\alpha}|{\rm d}x
	\leq\int_\Omega |f|\left[\left|e^{\frac{\alpha u}{2}}-1\right|^2+2\left|e^{\frac{\alpha u}{2}}-1\right|\right]{\rm d}x,~\forall N\geq 1 .
	\end{align}
	Then from (H1)-(H2), Lemma \ref{lian} and \eqref{1.2.2}, it infers
	\begin{align}\label{1.1}
	\int_\Omega \lambda(x)\left|\nabla F^{\alpha}_{N}(u)\right|^2 {\rm d}x
	&\leq \int_{\Omega}{\bf{A}}(x)\nabla u\nabla \phi^\alpha_N {\rm d}x\nonumber\\
	&\leq\int_\Omega |f|\left|e^{\frac{\alpha u}{2}}-1\right|^2 {\rm d}x+2\int_\Omega |f|\left|e^{\frac{\alpha u}{2}}-1\right|{\rm d}x,
	\end{align}
	where $F^{\alpha}_{N}=\displaystyle \int_{0}^{u}\left|\left(\phi_{N}^{\alpha}\right)^\prime(s)\right|^\frac{1}{2} {\rm d}s$. In addition, according to the definition of $\phi_{N}^{\alpha}$, we know
	\begin{align}\label{F}
	\varliminf \limits_{N\rightarrow \infty} \left| \nabla F^{\alpha}_{N}(u)\right|
	\geq \sqrt{\alpha} e^{\frac{\alpha u}{2}}|\nabla u|.
	\end{align}
	Therefore, together with H\"older's inequality, \eqref{F} and \eqref{1.1}, we have
	\begin{align}\label{sd}
	\left(\int_\Omega\left|\nabla \left(e^{\frac{\alpha u}{2}}-1\right)\right|^{\frac{2q}{q+1}}{\rm d}x\right)^{\frac{q+1}{2q}}
	&=\left(\int_\Omega \left|\nabla e^{\frac{\alpha u}{2}}\right|^{\frac{2q}{q+1}}{\rm d}x\right)^{\frac{q+1}{2q}}\nonumber\\
	&\leq \left\|\lambda^{-1}\right\|_{q}^\frac{1}{2} \left(\int_\Omega \lambda(x)\left|\nabla e^\frac{\alpha u}{2}\right|^2{\rm d}x\right)^\frac{1}{2}\nonumber\\
	&\leq \left\|\lambda^{-1}\right\|_{q}^\frac{1}{2} \left(\frac{\alpha}{4}\int_\Omega \lambda(x)\left|\nabla F^{\alpha}_{N}(u)\right|^2{\rm d}x\right)^\frac{1}{2}\nonumber\\
	&\leq \left\|\lambda^{-1}\right\|_{q}^\frac{1}{2}\left[\frac{\alpha}{4}\int_\Omega |f|\left|e^{\frac{\alpha u}{2}}-1\right|^2{\rm d}x\right.\notag\\
&\quad\left.+\frac{\alpha}{2}\int_\Omega |f|\left|e^\frac{\alpha u}{2}-1\right|{\rm d}x\right]^\frac{1}{2}.	
	\end{align}
	Once again, we apply Sobolev inequality and \eqref{sd} to observe that the following inequality holds
	\begin{align}\label{ty}
	\left(\int_\Omega\left|e^\frac{\alpha u}{2}-1\right|^{q^*}{\rm d}x\right)^{\frac{1}{q^*}}
	&\leq c_s\left(\int_\Omega\left|\nabla\left(e^\frac{\alpha u}{2}-1\right)\right|^{\frac{2q}{q+1}}{\rm d}x\right)^\frac{q+1}{2q}\nonumber\\
	&\leq c_s\left\|\lambda^{-1}\right\|_{q}^\frac{1}{2}\left[\frac{\alpha}{4}\left(\int_\Omega\left|e^\frac{\alpha u}{2}-1\right|^{q^*}{\rm d}x\right)^{\frac{2}{q^*}}\right.\notag\\
	&\quad\left.+\frac{\alpha}{2}\|f\|_{{(q^*)^\prime}}
	\left(\int_\Omega\left|e^\frac{\alpha u}{2}-1\right|^{q^*}{\rm d}x\right)^\frac{1}{q^*}\right]^\frac{1}{2},
	\end{align}
	where $q^*=\frac{2nq}{n(q+1)-2q}$.
	Meanwhile, we notice that
	\begin{align*}
	\|f\|_{(q^*)^\prime}\leq |\Omega|^{1-\frac{(q^*)^\prime}{s_0}},
	\end{align*}
      and $1-\frac{c_s\left\|\lambda^{-1}\right\|_q^\frac{1}{2}\sqrt{\alpha}}{2}>0$, so, it follows from \eqref{ty} that
	\begin{align}\label{tyu}
	\left(\int_\Omega\left|e^\frac{\alpha u}{2}-1\right|^{q^*}{\rm d}x\right)^\frac{1}{q^*}
	\leq \frac{2\alpha\|f\|_{(p^*)^\prime}}{\left(2-c_s\left\|\lambda^{-1}\right\|^\frac{1}{2}_{q}\sqrt{\alpha}\right)^2}\leq \frac{2\alpha|\Omega|^{1-\frac{(q^*)^\prime}{s_0}}}{\left(2-c_s\left\|\lambda^{-1}\right\|^\frac{1}{2}_{q}\sqrt{\alpha}\right)^2}.
	\end{align}
	By Minkowski's inequality and \eqref{tyu}, we get
	\begin{align}
	\left(\int_\Omega e^\frac{\alpha q^*u}{2}{\rm d}x\right)^\frac{1}{q^*}&\leq \left(\int_{\Omega}\left|e^{\frac{\alpha u}{2}}-1\right|^{q^*} {\rm d}x\right)^\frac{1}{q^*}+|\Omega|^\frac{1}{q^*}\nonumber\\
	&\leq \frac{2\alpha|\Omega|^{1-\frac{(q^*)^\prime}{s_0}}}{\left(2-c_s\left\|\lambda^{-1}\right\|^\frac{1}{2}_{q}\sqrt{\alpha}\right)^2}+|\Omega|^\frac{1}{q^*}.
	\end{align}
	This completes the proof of Lemma \ref{cs}.
\end{proof}

{\bf Proof of Theorem \ref{THM102}.}
Without any loss of generality, we assume $\sup\limits_{\Omega}u=\|u\|_{\infty}$, otherwise we consider $-u$.
Let $\varphi = \widetilde{\phi_{N}^{\alpha} }(u) \tilde{\eta}$, where
\begin{align}
\widetilde{\phi_{N}^{\alpha} } (u)=
\begin{cases}
e^{\frac{\alpha}{2}u_+}, & u\leq N,\\
\frac{\alpha}{2}e^{\frac{\alpha}{2}N}\left(u-N\right)+e^{\frac{\alpha}{2}N}, & u>N,
\end{cases}
~\text{and}~~\tilde{\eta}\geq 0\in C^\infty_0(\Omega).
\end{align}
We choose $\varphi$ as a test function in \eqref{def1} and obtain
\begin{align*}
\int_{\Omega}\left({\bf{A}}(x)\nabla u\nabla \widetilde{\phi_{N}^{\alpha} }\right)\tilde{\eta}{\rm d}x+\int_{\Omega}\left({\bf{A}}(x)\nabla u\nabla \tilde{\eta}\right)\widetilde{\phi_{N}^{\alpha}}{\rm d}x=\int_{\Omega}f\widetilde{\phi_{N}^{\alpha} } \tilde{\eta} {\rm d}x.
\end{align*}
In view of $\displaystyle \int_{\Omega}\left({\bf{A}}(x)\nabla u\nabla \widetilde{\phi_{N}^{\alpha} }\right)\tilde{\eta}{\rm d}x\geq 0$, we have
\begin{align}\label{qwe}
\int_{\Omega}\widetilde{\phi_{N}^{\alpha} }{\bf{A}}(x) \nabla u    \nabla \tilde{\eta}{\rm d}x\leq \int_{\Omega}f \widetilde{\phi_{N}^{\alpha} }\tilde{\eta} {\rm d}x.
\end{align}
Since $C^\infty_0(\Omega)$ is dense in $H^1_0(\Omega,{\bf{A}})$, it is obvious that \eqref{qwe} holds for any $\tilde{\eta}\in H^1_{0}(\Omega,{\bf{A}})$. Therefore, we may choose $\tilde{\eta}$ satisfying
\begin{align*}
\tilde{\eta}=\eta^\beta_{N} (u)=
\begin{cases}
\left(e^{\frac{\alpha}{2}u_+}\right)^\beta-1 , & u\leq N,\\
\frac{\alpha\beta}{2}e^{\frac{\alpha\beta}{2}N}(u-N)+e^{\frac{\alpha N}{2}}-1, & u>N,
\end{cases}
\end{align*}
where $\beta \ge 1$ as a test function. Hence, substituting $\eta^\beta_{N}$ into \eqref{qwe}, we have
\begin{align}\label{1.2}
\int_{\Omega}\left({\bf{A}}(x)\nabla u_+\nabla u_+\right)\left(\eta_N^\beta\right)^\prime \widetilde{\phi_N^\alpha}{\rm d}x\leq \int_{\Omega}f \widetilde{\phi_N^\alpha} \eta^\beta_{N}{\rm d}x.
\end{align}
Using Lemma \ref{lian} and (H1), we rewrite \eqref{1.2} as
\begin{align*}
\int_{\Omega}\lambda|\nabla\tilde{G}(u)|^2 {\rm d}x\leq \int_{\Omega}f\widetilde{\phi_N^\alpha} \eta^\beta_{N}{\rm d}x,
\end{align*}
where $\tilde{G}(u)=\int_{0}^{u}\left[\left(\eta_N^\beta\right)^\prime(s) \widetilde{\phi_N^\alpha}(s)\right]^\frac{1}{2}{\rm d}s$.
A simple computation shows
\begin{align}
&\varliminf\limits_{N\rightarrow \infty }|\nabla\tilde{G}(u)|^2\geq \frac{8\beta}{(\beta+1)^2}\left|\nabla\left(e^{\frac{\alpha}{2}u_+}\right)^\frac{\beta+1}{2}\right|^2,\label{TH2b1}\\
&\widetilde{\phi_N^\alpha} \eta^\beta_N\leq \left(e^{\frac{\alpha}{2}u_+}\right)^{\beta+1},~~\forall N\geq 1.\label{TH2b2}
\end{align}
Let $\omega=e^{\frac{\alpha}{2}u_+}$, then together with \eqref{1.2}-\eqref{TH2b2}, we obtain 
\begin{align}\label{1.3}
\int_{\Omega}\lambda\left|\nabla \omega^{\frac{\beta+1}{2}}\right|^2 {\rm d}x\leq \frac{\alpha(\beta+1)^2}{8\beta}\int_{\Omega}f\omega^{\beta+1} {\rm d}x.
\end{align}
With the aid of the Sobolev inequality and \eqref{1.3}, we deduce
\begin{align}\label{4r}
\left(\int_\Omega\left(\omega^\frac{\beta+1}{2}\right)^{q^*}{\rm d}x\right)^\frac{1}{q^*}
&\leq\left(\int_\Omega\left(\omega^\frac{\beta+1}{2}-1\right)^{q^*}{\rm d}x\right)^\frac{1}{q^*}+|\Omega|^\frac{1}{q^*}\nonumber\\
&\leq c_s\left\|\lambda^{-1}\right\|_{q}^\frac{1}{2}\left(\int_\Omega\lambda\left|\nabla \omega^\frac{\beta+1}{2}\right|^2 {\rm d}x\right)^\frac{1}{2}+|\Omega|^\frac{1}{q^*}\nonumber\\
&\leq c_s\left\|\lambda^{-1}\right\|_{q}^\frac{1}{2}\frac{\sqrt{\alpha}(\beta+1)}{2\sqrt{2\beta}}\|f\|_{s}^\frac{1}{2}\left(\int_\Omega\left(\omega^\frac{\beta+1}{2}\right)^{2s^\prime}{\rm d}x\right)^\frac{1}{2s^\prime}\nonumber\\
&\quad+|\Omega|^{\frac{1}{q^*}-\frac{1}{2s^\prime}}\left(\int_\Omega\left(\omega^\frac{\beta+1}{2}\right)^{2s^\prime}{\rm d}x\right)^\frac{1}{2s^\prime}\nonumber\\
&=\left[c_s\left\|\lambda^{-1}\right\|_{q}^\frac{1}{2}\frac{\sqrt{\alpha}\left(\beta+1\right)}{2\sqrt{2\beta}}\|f\|_{s}^\frac{1}{2}+|\Omega|^{\frac{1}{q^*}-\frac{1}{2s^\prime}}\right]\notag\\
&\quad\times\left(\int_\Omega\left(\omega^\frac{\beta+1}{2}\right)^{2s^\prime}{\rm d}x\right)^\frac{1}{2s^\prime},
\end{align}
where $2s^\prime<q^*$ due to $s>s_0=\left(\frac{q^*}{2}\right)^\prime$. Because of  $\frac{\beta+1}{2\sqrt{\beta}}\geq 1$, \eqref{4r} may be rewritten as the following form
\begin{align}\label{12d}
\left(\int_\Omega\left(\omega^\frac{\beta+1}{2}\right)^{q^*}{\rm d}x\right)^\frac{2}{q^*(\beta+1)}
&\leq\left[c_s\left\|\lambda^{-1}\right\|_q^\frac{1}{2}\frac{\sqrt{\alpha}(\beta+1)}{2\sqrt{2\beta}}\|f\|_{s}^\frac{1}{2}+|\Omega|^{\frac{1}{q^*}-\frac{1}{2s^\prime}}\right]^\frac{2}{\beta+1}\notag\\
&\quad\times\left(\int_\Omega\left(\omega^\frac{\beta+1}{2}\right)^{2s^\prime}{\rm d}x\right)^\frac{1}{s^\prime(\beta+1)}\nonumber\\
&\leq\left[C\frac{\sqrt{\alpha}(\beta+1)}{\sqrt{\beta}}\left(\left\|\lambda^{-1}\right\|_{q}^\frac{1}{2}\|f\|_{s}^\frac{1}{2}+1\right)\right]^\frac{2}{\beta+1}\notag\\
&\quad\times\left(\int_\Omega\left(\omega^\frac{\beta+1}{2}\right)^{2s^\prime}{\rm d}x\right)^\frac{1}{s^\prime(\beta+1)}.
\end{align}
Here $C=C(n,p,q,s,\Omega)\geq 1$. Now, if we set $\frac{\beta+1}{2}=\chi^{i},~i\in\mathbb{N}$ in \eqref{12d}, where $\chi=\frac{q^*}{2s^\prime}>1$, then \eqref{12d} becomes
\begin{align}
\|\omega\|_{{\chi^iq^*}}&\leq \left[\sqrt{\alpha}C\left(\left\|\lambda^{-1}\right\|_{q}^\frac{1}{2}\|f\|_{s}^\frac{1}{2}+1\right)\right]^\frac{1}{\chi^i}\chi^\frac{i}{\chi^i}\|\omega\|_{{2\chi^is^\prime}}\nonumber\\
&\leq \left[C\alpha \left(\left\|\lambda^{-1}\right\|_{q}\|f\|_{s}+1\right)\right]^{\sum_{j=0}^{i}\frac{1}{2\chi^j}}\chi^{\sum_{j=0}^{i}\frac{j}{\chi^j}}\|\omega\|_{{2s^\prime}}.
\end{align}
And then, taking $i\rightarrow +\infty$ yields
\begin{align}\label{int}
\|\omega\|_{\infty}\leq \left[C\alpha\left(\left\|\lambda^{-1}\right\|_{q}\|f\|_{{s}}+1\right)\right]^\frac{\chi}{2(\chi-1)}\chi^\frac{\chi}{(\chi-1)^2}\|\omega\|_{{2s^\prime}},
\end{align}
where $C=C(n,p,q,s,\Omega)$.
Using Lemma \ref{cs} and  H\"older's inequality,  \eqref{int} can be rewritten as
\begin{align*}
\|\omega\|_{{\infty}}\leq&\; C\left[\alpha\left(\left\|\lambda^{-1}\right\|_{q}\|f\|_{s}+1\right)\right]^\frac{\chi}{2(\chi-1)}\left(\int_\Omega \omega^{2s^\prime} {\rm d}x\right)^\frac{1}{2s^\prime}\\
\leq&\; C\left[\alpha\left(\left\|\lambda^{-1}\right\|_{q}\|f\|_{s}+1\right)\right]^\frac{\chi}{2(\chi-1)}\left(\left(\int_{\Omega}e^{\frac{\alpha}{2}q^* u}{\rm d}x\right)^\frac{1}{q^*}+|\Omega|^\frac{1}{q^*}\right)\\
\leq&\; C\left[\alpha\left(\left\|\lambda^{-1}\right\|_{q}\|f\|_{s}+1\right)\right]^\frac{\chi}{2(\chi-1)}\left(\frac{\alpha |\Omega|^{1-\frac{(q^*)^\prime}{s_0}}}{\left(2-c_s\left\|\lambda^{-1}\right\|_{q}\sqrt{\alpha}\right)^2}+|\Omega|^\frac{1}{q^*}\right).
\end{align*}
By the definition of $\omega$, we further obtain
\begin{align}\label{ite}
e^{\frac{\alpha}{2}\|u_+\|}\leq C\left[\alpha\left(\|\lambda^{-1}\|_{q}\|f\|_{s}+1\right)\right]^\frac{\chi}{2(\chi-1)}\left(\frac{\alpha |\Omega|^{1-\frac{(q^*)^\prime}{s_0}}}{\left(2-c_s\|\lambda^{-1}\|_{q}\sqrt{\alpha}\right)^2}+|\Omega|^\frac{1}{q^*}\right).
\end{align}
Obviously, \eqref{ite} is equivalent to 
\begin{align}
\|u\|_{\infty}\leq C\left[\log\left(\left\|\lambda^{-1}\right\|_q\|f\|_{s}+1\right)+1\right],
\end{align}
where $C=C(n,p,q,s,\Omega)\geq 1$.
This completes the proof.
\begin{remark}
\rm We need to point out some differences below. Firstly, in the proof of Lemma \ref{cs}, we construct a truncation function that satisfies the condition of Lemma \ref{lian} to bypass some difficulties from the global boundedness of $u$, as described in [\ref{CUR}, \ref{XU}]. Secondly, In the proof of Theorem \ref{THM102}, we construct a Caccioppoli inequality for $e^{\alpha u_+}$ to avoid the transformation in [\ref{CUR}, \ref{XU}] which essentially depends on the bounded condition.
\end{remark}

\section{Harnack Inequality~---~Proof of Theorem \ref{THM103}}
To end the proof of Theorem \ref{THM103}, we first give a key lemma, namely, weak Harnack inequality.
\begin{lemma}\label{THM104}
	(Weak Harnack inequality) Assume that (H1)-(H2) hold, $q>\frac{n}{2}$ and $f\in L^s(\Omega)$ for some $s>\frac{nq}{2q-n}$.
	Let $u\in H^1({\bf{A}},\Omega)$ be a nonnegative supersolution in $\Omega$. Then for any $B_R\subset \Omega$, $\gamma\in (0,\frac{q_*}{2}) $ and any $0<\theta<\tau<1$, there exists a constant $C$ such that
	\begin{align*}
	\left(\frac{1}{R^n}\int_{B_{\tau R}} u^\gamma {\rm d}x \right)^\frac{1}{\gamma}\leq C\left(\inf_{B_{\theta R}}u +R^{2-\frac{n}{s}}\|f\|_{L^s(B_R)}\right),
	\end{align*}
	where $q^*=\frac{2nq}{n(q+1)-2q}$, $C\leq c_1 e^{c_2\Lambda\left(B_R\right)}$, $c_1=c_1\left(\gamma,n,p,q,s,\tau,\theta\right)\geq 1$, $c_2=c_2\left(\gamma,n,p,q\right)>0$.
\end{lemma}
\begin{remark}
	\rm
	Theorem \ref{THM103} may be concluded by combining Theorem \ref{THM101} with Lemma \ref{THM104}.
\end{remark}
The prove of Lemma \ref{THM104} is based on the strategy of [\ref{BPMS3}, Theorem 4.1]. Even though experts might already anticipate how to adapt the argument of [\ref{BPMS3}], for the completeness of the article, we give a detailed proof. 

{\bf Proof of Lemma \ref{THM104}.}
To better comprehend the structure of our proof, we will divide our proof into three steps. And  without loss of generality, we set $R=1$ here.

{\bf Step 1.} We claim that
\begin{align}\label{ine1}
\exp\left(\fint_{B_{\tau}}\log(u+h){\rm d}x\right)\leq C\left(\inf_{B_{\theta }}u+\|f\|_{L^s(B_1)}\right),
\end{align}
where $C=C\left(n,\tau,\theta,\Lambda(B_1),p,q,s\right)$.

First, we take $\phi=\frac{\varphi}{u+h}$ in \eqref{def1}, where $h:=\|f\|_{L^s(B_1)}$ and $\varphi\geq 0\in H^1_0(B_1,{\bf{A}})$, then we have
\begin{align}\label{yh}
\int_{B_1}\left( \frac{1}{u+h}{\bf {A}}(x)\nabla u\nabla \varphi -\frac{\varphi}{(u+h)^2} {\bf {A}}(x)\nabla u\nabla u\right)  {\rm d}x \geq \int_{B_1}\frac{f\varphi}{u+h} {\rm d}x.
\end{align}
Due to $\frac{\varphi}{(u+h)^2}{\bf {A}}(x)\nabla u\nabla u\geq 0$, \eqref{yh} becomes
\begin{align}\label{logx}
\int_{B_1}{\bf {A}}(x)\nabla \left(\log\frac{k}{u+h}\right)\nabla \varphi{\rm d}x\leq -\int_{B_1}\frac{f\varphi}{u+h}{\rm d}x.
\end{align}
Set $\log\frac{k}{u+h}=v$ and $\frac{f}{u+h}=\tilde{f}$, and rewrite \eqref{logx} as
\begin{align}\label{6u}
\int_{B_1}{\bf{A}}(x)\nabla v\nabla \varphi {\rm d}x\leq -\int_{B_1}\tilde{f}\varphi {\rm d}x,
\end{align}
where $\|\tilde{f}\|_{L^s(B_1)}\leq 1$. Then \eqref{6u} implies that $v$ is the weak subsolution of $-\dive({\bf{A}}(x)\nabla v)=-\tilde{f}$. 

To give a priori estimate for $v$, we choose $\varphi=\varphi_m\eta^2$ in \eqref{6u}, where
\begin{align*}
\varphi_m=
\begin{cases}
(v_++1)^\beta-1,& v<m,\\
\beta(m+1)^{\beta-1}(v-m)+(m+1)^\beta-1,&v\geq m,
\end{cases}
\end{align*}
and $\eta\in C^\infty_0(B_{\tau}),~\eta\equiv 1 $ in $B_{\rho }$; $\eta\equiv 0 $ in $B_\tau \setminus B_{\sigma}$ and $0<\theta\leq\rho<\sigma\leq \tau\leq 1$, by the same procedure of Theorem \ref{THM101}, we have
\begin{align}\label{ruo1}
\sup_{B_{\theta }}v\leq C \left(\Lambda(B_1)\right)^{\frac{\delta p^\prime}{q_*(\delta-1)}}\left(\frac{1}{(\tau-\theta)^{\frac{n}{q_*}(m_*+1)}}\|v\|_{L^{q_*}(B_{\tau })}+\frac{1}{\left(\tau-\theta\right)^{\frac{n m_*}{q^*}+\frac{n}{s}-2}}\right),
\end{align}
where $q^*=\frac{2nq}{n(q+1)-2q}$.
In the sequel, we will estimate the term $\|v\|_{L^{q_*}(B_{\tau })}$ in \eqref{ruo1}. Let $\phi=\frac{\eta_1^2}{u+h}$ in \eqref{def1}, where $\eta_1\in C^\infty_0(B_{1})$ and $\eta_1\equiv 1$ in $B_{\tau}$; $\eta_1 \equiv 0$ in $B_1\setminus B_{\frac{1+\tau}{2}}$. After direct calculation, we have
\begin{align*}
\int_{B_1}\eta_1^2 {\bf{A}}(x)\nabla v\nabla v{\rm d}x +\int_{B_1}2\eta_1 {\bf{A}}(x)\nabla v\nabla \eta_1 {\rm d}x\leq -\int_{B_1}\tilde{f}\eta^2_1 {\rm d}x,
\end{align*}
and then using the Young's inequality and H\"older's inequality yields
\begin{align}\label{ineb}
\int_{B_1}\lambda |\nabla v|^2\eta_1^2 {\rm d}x&\leq C \int_{B_1}\mu |\nabla\eta_1|^2{\rm d}x+|B_1|^{\frac{1}{s^\prime}} \notag\\
&\leq C \frac{\|\mu\|_{L^p(B_1)}}{(1-\tau)^2} |B_{1}|^\frac{1}{p^\prime}+|B_1|^{\frac{1}{s^\prime}} .
\end{align}
Subsequently, we let $k=\exp\left(\fint_{B_{\tau}}\log \left(u+h\right){\rm d}x \right)$ satisfying $\fint_{B_\tau}v {\rm d}x =0$,  and apply Poincar$\acute{e}$ inequality to get
\begin{align}\label{inea}
\|v\|_{L^{q_*}(B_{\tau })}\leq C\|\nabla v\|_{L^\frac{2q}{q+1}(B_{\tau})}\leq \|\lambda^{-1}\|_{L^q(B_{\tau })}^\frac{1}{2}\|\lambda^{\frac{1}{2}}\nabla v\|_{L^2(B_{\tau})}.
\end{align}
Substituting \eqref{ineb} into \eqref{inea} and it is not hard to verify 
\begin{align*}
\sup_{B_{\theta }}v\leq C\left(\Lambda(B_1)\right)^{\frac{\delta p^\prime}{q_*(\delta-1)}+\frac{1}{2}}\left(\frac{1}{(\tau-\theta)^{\frac{n}{q_*}(m_*+1)}(1-\tau)}+\frac{1}{\left(\tau-\theta\right)^{\frac{n m_*}{q_*}+\frac{n}{s}-2}}\right).
\end{align*}
Finally, \eqref{ine1} follows from the definitions of $v$ and $k$.

{\bf Step 2.}  We claim that there exist $p_0=p_0(\tau,\theta,\Lambda(B_1),p,s,n,q)>0$ and $C=C(\tau,\theta,\Lambda(B_1),$ $p,s,n,q)>0$ such that
\begin{align}\label{ine2}
\left(\int_{B_{\frac{3\tau+\theta}{4}}}|u+h|^{p_0}{\rm d}x\right)^\frac{1}{p_0}\leq C\exp\left(\fint_{B_{\tau}}\log (u+h){\rm d}x\right).
\end{align}

First of all, we take $\phi=\eta^2\frac{1}{u+h}\left((\omega_++1)^\beta+(2\beta)^{\beta}\right)$ in \eqref{def1}, where $\omega=\log\frac{u+h}{k}=-v,~\beta\geq1$,
$\eta\equiv 1$ in $B_\rho$; $\eta_1 \equiv 0$ in $B_\tau\setminus B_{\sigma}$ and $\eta\in C^\infty_0(B_\sigma)$, with $\frac{3\tau+\theta}{4}\leq \rho<\sigma\leq \tau<1$,
and notice
\begin{align*}
\nabla \phi&=\frac{2\eta}{u+h}\left((\omega_++1)^\beta+(2\beta)^\beta\right)\nabla \eta-\frac{\eta^2}{(u+h)^2}\left((\omega_++1)^\beta+(2\beta)^\beta\right)\nabla u\\
&\quad+\frac{\beta\eta^2}{u+h}\left(\omega_++1\right)^{\beta-1}\nabla (\omega_++1).
\end{align*}
Then, we obtain
\begin{align}\label{ine3}
&\int_{B_\sigma}\frac{\eta^2}{(u+h)^2}\left((\omega_++1)^\beta+(2\beta)^{\beta}\right){\bf{A}}(x)\nabla u\nabla u {\rm d}x \notag\\
&\quad -\beta\int_{B_\sigma}\frac{\eta^2}{u+h}\left(\omega_++1\right)^{\beta-1}{\bf{A}}(x)\nabla u\nabla\left(\omega_++1\right) {\rm d}x \notag\\
\leq&\int_{B_\sigma}\frac{2\eta}{u+h}\left(\left(\omega_++1\right)^\beta+\left(2\beta\right)^\beta\right){\bf{A}}(x)\nabla u\nabla\eta {\rm d}x-\int_{B_\sigma}f\frac{\eta^2}{u+h}\left((\omega_++1)^\beta+(2\beta)^{\beta}\right){\rm d}x .
\end{align}
Subsequently, using Young's inequality, we can estimate the first term on the right-hand side of \eqref{ine3} as follows
\begin{align}\label{ine4}
&\int_{B_\sigma}\frac{2\eta}{u+h}\left(\left(\omega_++1\right)^\beta+\left(2\beta\right)^\beta\right){\bf{A}}(x)\nabla u\nabla\eta {\rm d}x\notag\\
\leq&\;\frac{1}{4}\int_{B_\sigma}\frac{\eta^2\left(\omega_++1\right)^{\beta-1}}{\left(u+h\right)^2}{\bf{A}}(x)\nabla u\nabla u{\rm d}x +8\int_{B_\sigma}\left(\omega_++1\right)^{\beta+1}{\bf{A}}(x)\nabla\eta\nabla \eta {\rm d}x\notag\\
& +\frac{1}{4}\int_{B_\sigma}\frac{\eta^2}{(u+h)^2}\left(2\beta\right)^\beta {\bf{A}}(x)\nabla u\nabla u{\rm d}x +8\int_{B_\sigma}\left(2\beta\right)^\beta {\bf{A}}(x)\nabla\eta\nabla\eta {\rm d}x\notag\\
\leq&\;\frac{1}{4}\int_{B_\sigma}\frac{\eta^2\left(\omega_++1\right)^{\beta}}{\left(u+h\right)^2}{\bf{A}}(x)\nabla u\nabla u{\rm d}x+8\int_{B_\sigma}\left(\omega_++1\right)^{\beta+1}{\bf{A}}(x)\nabla\eta\nabla \eta {\rm d}x\notag\\
& +\frac{1}{4}\int_{B_\sigma}\frac{\eta^2}{(u+h)^2}\left(2\beta\right)^\beta {\bf{A}}(x)\nabla u\nabla u{\rm d}x +8\int_{B_\sigma}\left(2\beta\right)^\beta {\bf{A}}(x)\nabla\eta\nabla\eta{\rm d}x\notag\\
=&\;\frac{1}{4}\int_{B_\sigma}\frac{\eta^2}{(u+h)^2}\left((\omega_++1)^\beta+(2\beta)^{\beta}\right){\bf{A}}(x)\nabla u\nabla u {\rm d}x\notag\\
& +8\int_{B_\sigma}\left(\omega_++1\right)^{\beta+1}{\bf{A}}(x)\nabla\eta\nabla \eta {\rm d}x+ 8\int_{B_\sigma}\left(2\beta\right)^\beta {\bf{A}}(x)\nabla\eta\nabla\eta {\rm d}x.
\end{align}
By definition of  $\omega_+$, we observe that the first term on the right-hand side in \eqref{ine4} can be absorbed into the first term on the left-hand side of \eqref{ine3}. This leads to the following inequality
\begin{align*}
&\int_{B_\sigma}\eta^2\left[\frac{3}{4}\left(\left(\omega_++1\right)^\beta+\left(2\beta\right)^\beta\right)-\beta\left(\omega_++1\right)^{\beta-1}\right]{\bf{A}}(x)\nabla \left(\omega_++1\right)\nabla \left(\omega_++1\right){\rm d}x \\
\leq 	&\;8\int_{B_{\sigma}}\left(\omega_++1\right)^{\beta+1} {\bf{A}}(x)\nabla\eta\nabla\eta {\rm d}x+8\int_{B_\sigma}\left(2\beta\right)^\beta {\bf{A}}(x)\nabla \eta\nabla\eta {\rm d}x\\
&+\int_{B_\sigma}|\tilde{f}|\eta^2\left(\left(\omega_++1\right)^\beta+\left(2\beta\right)^\beta\right){\rm d}x.
\end{align*}
Owing to $\left(\omega_++1\right)^\beta+\left(2\beta\right)^\beta\geq 2\beta\left(\omega_++1\right)^{\beta-1}$, the previous inequality reduces
\begin{align}\label{r3}
&\int_{B_\sigma}\beta\eta^2\left(\omega_++1\right)^{\beta-1}{\bf{A}}(x)\nabla \left(\omega_++1\right)\nabla \left(\omega_++1\right){\rm d}x \nonumber\\
\leq& \; 16\int_{B_{\sigma}}\left(\omega_++1\right)^{\beta+1} {\bf{A}}(x)\nabla\eta\nabla\eta {\rm d}x
+16\;\int_{B_\sigma}\left(2\beta\right)^\beta {\bf{A}}(x)\nabla \eta\nabla\eta {\rm d}x\nonumber\\
&+2\int_{B_\sigma}|\tilde{f}|\eta^2\left(\left(\omega_++1\right)^\beta+\left(2\beta\right)^\beta\right){\rm d}x.
\end{align}
With the help of ({\rm H2}) and H\"older's inequality in \eqref{r3}, we obtain
\begin{align}\label{ine5}
\frac{4\beta}{\left(\beta+1\right)^2}\int_{B_\sigma}\lambda\left|\nabla\left(\omega_++1\right)^{\frac{\beta+1}{2}}\right|^2\eta^2 {\rm d}x
&\leq 16\int_{B_\sigma}\mu|\nabla\eta|^2\left(\omega_++1\right)^{\beta+1} {\rm d}x \notag\\
&\quad+ \left(2\beta\right)^\beta\left(16\|\mu\|_{L^p(B_\sigma)}\left|B_1\right|^\frac{1}{p^\prime}\frac{1}{(\sigma-\rho)^2}+2\left|B_1\right|^\frac{1}{s^\prime}\right)\notag\\
&\quad+2\int_{B_\sigma}|\tilde{f}|\eta^2\left(\omega_++1\right)^{\beta+1} {\rm d}x.
\end{align}
Hence, following the lines of proof of \eqref{dai1}-\eqref{dai3}, we conclude
\begin{align}\label{ine6}
\int_{B_\sigma}|\tilde{f}|\eta^2\left(\omega_++1\right)^{\beta} {\rm d}x
\leq &\left(\int_{B_\sigma}\left|\eta\left(\omega_++1\right)^\frac{\beta+1}{2}\right|^{2s^\prime} {\rm d}x\right)^\frac{1}{s^\prime}\notag\\
&\leq \varepsilon \left(\int_{B_\sigma}\left|\nabla\left(\left(\omega_++1\right)^\frac{\beta+1}{2}\eta\right)\right|^\frac{2q}{q+1}{\rm d}x\right)^\frac{q+1}{q}\notag\\
&\quad+\frac{1}{\varepsilon}\left(\int_{B_\sigma}\left|\left(\omega_++1\right)^\frac{\beta+1}{2}\eta\right|^{p^*}{\rm d}x\right)^\frac{2}{p^*}\notag\\
&\leq 2\varepsilon\|\lambda^{-1}\|_q\left[\int_{B_\sigma}\lambda\left|\nabla\left(\omega_++1\right)^\frac{\beta+1}{2}\right|^2\eta^2{\rm d}x\right.\notag\\
&\quad\left.+\int_{B_\sigma}\mu\left|\nabla\eta\right|^2\left(\omega_++1\right)^{\beta+1}{\rm d}x\right]\notag\\
&\quad+\frac{1}{\varepsilon}\left(\int_{B_\sigma}\left|\left(\omega_++1\right)^\frac{\beta+1}{2}\eta\right|^{p^*}{\rm d}x\right)^\frac{2}{p^*}.
\end{align}
By setting $\varepsilon=\frac{\beta}{2\left(\beta+1\right)^2\|\lambda^{-1}\|_{L^q(B_\sigma)}}$ and absorbing the first term on the right-hand side of \eqref{ine6} into the left-hand side of \eqref{ine5},   it is easy to calculate that
\begin{align}\label{omec}
\int_{B_\sigma}\lambda\left|\nabla \left(\omega_++1\right)^\frac{\beta+1}{2}\right|^2\eta^2 {\rm d}
&\leq \frac{16\left(\beta+1\right)^2}{\beta}\int_{B_\sigma}\mu|\nabla\eta|^2\left(\omega_++1\right)^{\beta+1}{\rm d}x\nonumber\\
&\quad +\frac{2\|\lambda^{-1}\|_q\left(\beta+1\right)^4}{\beta^2}\left(\int_{B_\sigma}\left|\left(\omega_++1\right)^\frac{\beta+1}{2}\eta\right|^{p^*}{\rm d}x\right)^\frac{2}{p^*}\nonumber\\
& \quad +32\left(\beta+1\right)^2\left(2\beta\right)^{\beta-1}\|\mu\|_{L^p(B_\sigma)}\frac{1}{\left(\sigma-\rho\right)^2}.
\end{align}
Furthermore, minimizing the right hand of \eqref{omec} for $\eta(|x|)\in C^\infty_0(B_\tau)$, with $\eta(|x|)\equiv1$ in $B_\rho$; $\eta(|x|)\equiv0$ in $B_\tau \setminus B_\sigma$, and using Lemma \ref{lem2.1}, we obtain
\begin{align}\label{tg}
\int_{B_\sigma}\lambda\left|\nabla \left(\omega_++1\right)^\frac{\beta+1}{2}\right|^2\eta^2{\rm d}x
&\leq \left[\frac{16\left( \beta+1\right)^2}{\beta\left(\sigma-\rho\right)^{\frac{2n}{n-1}}}\|\mu\|_{L^p(B_\sigma)}+2\|\lambda^{-1}\|_{L^q(B_\sigma)}\frac{\left(\beta+1\right)^4}{\beta^2}\right]\nonumber\\
&\quad\times\left\|\left(\omega_++1\right)^\frac{\beta+1}{2}\right\|^2_{W^{1,p^*}(B_\sigma)}\notag\\
&\quad+32\left(\beta+1\right)^2\left(2\beta\right)^{\beta-1}\|\mu\|_{L^p(B_\sigma)}\frac{1}{\left(\sigma-\rho\right)^2}.
\end{align}
Then, we replace $u$ by $\omega$ in inequalities \eqref{cdiedai}-\eqref{qia} and use \eqref{tg} to get
\begin{align}\label{ome2}
\left\|\left(\omega_++1\right)^{\frac{\beta+1}{2}\delta}\right\|_{W^{1,p^*}(B_\rho)}
&\leq C \Lambda(B_1)^\frac{1}{2}\left(\frac{\beta+1}{\left(\sigma-\rho\right)^\frac{n}{n-1}\sqrt{\beta}}+\frac{\left(\beta+1\right)^2}{\beta}\right)\left\|\left(\omega_++1\right)^\frac{\beta+1}{2}\right\|^\delta_{W^{1,p^*}(B_\sigma)}\nonumber\\
&\quad+\left(2\beta\right)^\frac{\beta-1}{2}\left(\beta+1\right)\Lambda(B_1)^\frac{1}{2}\frac{1}{\sigma-\rho}\left(\int_{B_\rho}\left|\left(\omega_++1\right)^\frac{\beta+1}{2}\right|^{p^*}{\rm d}x\right)^\frac{\delta-1}{p^*}.
\end{align}
Keeping Young's inequality in mind, we find that the second term on the right-hand side of \eqref{ome2} has the following estimate
\begin{align}\label{ome3}
\left\|\left(\omega_++1\right)^{\frac{\beta+1}{2}\delta}\right\|_{W^{1,p^*}(B_\rho)}&\leq C \Lambda(B_1)^\frac{\delta}{2\left(\delta-1\right)}\left(\frac{\beta+1}{\left(\sigma-\rho\right)^\frac{n}{n-1}\sqrt{\beta}}+\frac{\left(\beta+1\right)^2}{\beta}+\left(\sigma-\rho\right)^\frac{-\delta}{\delta-1}\right)\nonumber\\
&\quad\times\left\|\left(\omega_++1\right)^\frac{\beta+1}{2}\right\|^\delta_{W^{1,p^*}(B_\sigma)}
+\left[\left(2\beta\right)^\frac{\beta-1}{2}\left(\beta+1\right)\right]^\delta.
\end{align}
Set $\alpha=\frac{\beta+1}{2}\geq1$, then, \eqref{ome3} is rewritten as
\begin{align*}
\left\|\left(\omega_++1\right)^{\alpha\delta}\right\|^\frac{1}{\alpha\delta}_{W^{1,p^*}\left(B_\rho\right)}
&\leq C^\frac{1}{\alpha\delta}\Lambda(B_1)^\frac{1}{2\alpha\left(\delta-1\right)}\left(\frac{\alpha}{\left(\sigma-\rho\right)^\frac{n}{n-1}\sqrt{2\alpha-1}}+\frac{\alpha^2}{2\alpha-1}+\left(\sigma-\rho\right)^\frac{\delta}{\delta-1}\right)^\frac{1}{\delta\alpha}\\
&\quad\times\left\|\left(\omega_++1\right)^\alpha\right\|^\frac{1}{\alpha}_{W^{1,p^*}(B_\sigma)}+4\alpha.
\end{align*}
Therefore, we set $\alpha=\delta^m$,  $\rho_m=\frac{3\tau+\theta}{4}+\frac{1}{2^{m+2}}\left(\tau-\theta\right)$, $m \in \mathbb{N}$, and build up the following iterative process
\begin{align}\label{ine7}
\left\|\left(\omega_++1\right)^{\delta^{m+1}}\right\|^\frac{1}{\delta^{m+1}}_{W^{1,p^*}(B_{\rho_{m+1}})}
&\leq C^\frac{1}{\delta^{m+1}}\Lambda(B_1)^\frac{1}{2\left(\delta-1\right)\delta^m}\left(2^{\frac{\delta}{\delta-1}+\frac{n}{n-1}}\delta\right)^{\frac{m+1}{\delta^m}}\notag\\
&\quad\times\left\|\left(\omega_++1\right)^{\delta^m}\right\|^\frac{1}{\delta^m}_{W^{1,p^*}(B_{\rho_m})}+4\delta^m,
\end{align}
which implies
\begin{align}\label{did}
\left\|\left(\omega_++1\right)^{\delta^{m+1}}\right\|^\frac{1}{\delta^{m+1}}_{W^{1,p^*}(B_{\rho_{m+1}})}
&\leq C\Lambda(B_1)^\frac{\delta}{2\left(\delta-1\right)^2}\left(2^{\frac{\delta}{\delta-1}+\frac{n}{n-1}}\delta\right)^\frac{\delta}{\left(\delta-1\right)^2}\notag\\
&\quad\times\left(\left\|\omega_++1\right\|_{W^{1,p^*}(B_{\rho_0})}+\delta^m\right),
\end{align}
where $C$ depends on $n,~p,~q,~\tau,~\theta$.

Finally, recall that $\omega=-v$, together with   \eqref{ineb} and \eqref{inea}, and we get
\begin{align*}
\|\omega\|_{W^{1,p^*}(B_{\tau})}=\|v\|_{W^{1,p^*}(B_{\tau})}\leq C \|v\|_{W^{1,\frac{2q}{q+1}}(B_{\tau})}\leq C(\tau,~\Lambda(B_1),~p,~q, s).
\end{align*}
Then for any $i> p^*$, there exists $m_0\in \mathbb{N}$ such that $\delta^{m_0}p^*\leq i\leq \delta^{m_0+1}p^*$, \eqref{did} can be rewritten as
\begin{align*}
\left\|\omega_+\right\|_{L^i\left(B_\frac{3\tau+\theta}{4}\right)}\leq C_1\Lambda(B_1)^{C_2}i,
\end{align*}
where $C_1=C_1(\tau,p,q,s,n,\theta)$ and $C_2=C_2(n,p,q)$.
We further let $p_0=\left(2C_1\Lambda(B_1)^{C_2}e\right)^{-1}$  satisfying
\begin{align*}
\frac{p_0^i\|\omega_+\|^i_{L^i(B_\frac{3\tau+\theta}{4})}}{i!}\leq \frac{1}{2^i},
\end{align*}
thus we have
\begin{align*}
\left(\int_{B_{\frac{3\tau+\theta}{4}}}\left(\frac{u+h}{k}\right)^{p_0} {\rm d}x\right)^\frac{1}{p_0}
&\leq \left(\int_{B_\frac{3\tau+\theta}{4}}e^{p_0\omega_+}{\rm d}x +|B_1| \right)^\frac{1}{p_0}\\
&\leq \left(\sum_{i=0}^{\infty}\frac{p_0\|\omega_+\|_{L^i(B_\frac{3\tau+\theta}{4})}^i}{i!}+|B_1|\right)^\frac{1}{p_0}
\leq \left(2+|B_1|\right)^\frac{1}{p_0}.
\end{align*}

The proof of the inequality \eqref{ine2} is complete.

{\bf Step 3.} We claim that for any $\gamma\in [p_0,\frac{q_*}{2})$, there exists $C=C(p,q, \Lambda(B_1), \gamma,n)\geq 1$ such that
\begin{align}\label{5y}
\left(\int_{B_{\theta}}\left(u+h\right)^{\gamma}{\rm d}x\right)^\frac{1}{\gamma}\leq C\left(\int_{B_\frac{3\tau+\theta}{4}}\left(u+h\right)^{p_0}{\rm d}x\right)^\frac{1}{p_0},
\end{align}
where $C\leq C_3\Lambda\left(B_1\right)^{C_4\left(\frac{1}{p_0}-\frac{1}{\gamma}\right)} $ with $C_3=C_3(\gamma, n,p,q,\tau,\theta)\in[1,+\infty)$ and $C_4=C_4(\gamma,n,p,q)\in [1,+\infty)$.

To complete the proof of \eqref{5y}, we take  $\phi=\eta^2\left(u+h\right)^\beta$ with $\beta\in (-1, 0)$ in \eqref{def1}, and $0<\frac{\tau+3\theta}{4}\leq \rho<\sigma\leq \frac{\tau+\theta}{2}<1$, $\eta\equiv 1$ in $B_\rho$; $\eta_1 \equiv 0$ in $B_{\frac{\tau+\theta}{2}}\setminus B_\sigma$, we obtain
\begin{align*}
\beta\int_{B_\sigma}\left(u+h\right)^{\beta-1}\eta^2 {\bf{A}}(x)\nabla u\nabla\left(u+h\right){\rm d}x
+2\int_{B_\sigma}\eta\left(u+h\right)^\beta {\bf{A}}(x)\nabla u\nabla\eta {\rm d}x
\geq \int_{B_\sigma}f\eta^2\left(u+h\right)^\beta {\rm d}x,
\end{align*}
then by Young's inequality, we get
\begin{align}\label{ine10}
-\beta\int_{B_\sigma}\left(u+h\right)^{\beta-1}\eta^2 {\bf{A}}(x)\nabla u\nabla\left(u+h\right) {\rm d}x
&\leq -\int_{B_\sigma}f\eta^2\left(u+h\right)^\beta {\rm d}x\notag\\
&\quad+2\int_{B_\sigma}\eta\left(u+h\right)^\beta {\bf{A}}(x)\nabla u\nabla\eta {\rm d}x\notag\\
&\leq -\frac{\beta}{2}\int_{B_\sigma}\left(u+h\right)^{\beta-1}\eta^2 {\bf{A}}(x)\nabla u\nabla\left(u+h\right) {\rm d}x\notag\\
&\quad  -\int_{B_\sigma}f\eta^2\left(u+h\right)^\beta {\rm d}x\notag\\
&\quad-\frac{8}{\beta}\int_{B_\sigma}\mu\left|\nabla\eta\right|^2\left(u+h\right)^{\beta+1}{\rm d}x.
\end{align}
Absorbing the first term on the right-hand side into the left-hand side of \eqref{ine10}, we obtain
\begin{align*}
-\frac{4\beta}{\left(\beta+1\right)^2}\int_{B_\sigma}\lambda\left|\nabla \left(u+h\right)^\frac{\beta+1}{2}\right|^2 {\rm d}x
&\leq \int_{B_\sigma}|f|\eta^2\left(u+h\right)^\beta {\rm d}x\notag\\
&\quad-\frac{8}{\beta}\int_{B_\sigma}\mu\left|\nabla\eta\right|^2\left(u+h\right)^{\beta+1}{\rm d}x.
\end{align*}
Next, recalling \eqref{ine6} and replacing  $\omega_++1$ by $u+h$, the above inequality may be  rewritten as
\begin{align}\label{ine11}
\int_{B_\sigma}\lambda\left|\nabla \left(u+h\right)^\frac{\beta+1}{2}\right|^2 {\rm d}x
&\leq \frac{16\left(\beta+1\right)^2}{|\beta|}\int_{B_\sigma}\mu\left|\nabla\eta\right|^2\left(u+h\right)^{\beta+1}{\rm d}x\notag\\
&\quad+\frac{2\|\lambda^{-1}\|_q\left(\beta+1\right)^4}{\beta^2}\left(\int_{B_\sigma}\left(u+h\right)^\frac{\left(\beta+1\right)p_*}{2}{\rm d}x\right)^\frac{2}{p^*}.
\end{align}
In turn, keeping Lemma \ref{lem2.1} and H\"older's inequality in mind, we get
\begin{align}\label{zuih}
\left\|\nabla\left(u+h\right)^\frac{\beta+1}{2}\right\|_{L^\frac{2q}{q+1}\left(B_\rho\right)}
\leq C\left(\frac{\beta+1}{\sqrt{\beta}}+\frac{\left(\beta+1\right)^2}{|\beta|}\right)\frac{\Lambda(B_1)^\frac{1}{2}}{\left(\sigma-\rho\right)^\frac{n}{n-1}}\left\|\left(u+h\right)^
\frac{\beta+1}{2}\right\|_{W^{1,p_*}\left(B_\sigma\right)}.
\end{align}
Set $\alpha=\frac{\beta+1}{2}\in \left(0,\frac{1}{2}\right)$, then \eqref{dao}, \eqref{qia} and \eqref{zuih} imply
\begin{align*}
\left\|\left(u+h\right)^{\alpha\delta}\right\|_{W^{1,p_*}(B_\rho)}^{\frac{1}{\alpha\delta}}
&\leq \frac{\left(C\Lambda(B_1)^\frac{1}{2}\right)^\frac{1}{\alpha\delta}}{\left(\sigma-\rho\right)^\frac{n}{\left(n-1\right)\alpha\delta}}\left(\frac{2\alpha}{\sqrt{1-2\alpha}}+\frac{4\alpha^2}{1-2\alpha}\right)^\frac{1}{\alpha\delta}\left\|\left(u+h\right)^\alpha\right\|^\frac{1}{\alpha}_{W^{1,p_*}(B_\sigma)}.
\end{align*}
Taking $\alpha=\frac{\alpha_0}{\delta^m}$ where $\alpha_0\in\left(0,\frac{1}{2}\right)$, and
$\rho_m=\frac{\tau+\theta}{2}+\frac{1}{2^m}\frac{\left(\theta-\tau\right)}{4}$, we get
\begin{align*}
\left\|\left(u+h\right)^\frac{\alpha_0}{\delta^{m-1}}\right\|^\frac{\delta^{m-1}}{\alpha_0}_{W^{1,p^*}\left(B_{\rho_{m-1}}\right)}
&\leq \left(\frac{C\Lambda(B_1)^\frac{1}{2}}{\left(\tau-\theta\right)}\right)^\frac{\delta^{m-1}}{\alpha_0}\left(\frac{\alpha_0}{\delta^m\sqrt{1-\frac{2\alpha_0}{\delta^m}}}+\frac{4\alpha_0^2}{\delta^{2m}\left(1-\frac{2\alpha_0}{\delta^m}\right)}\right)^\frac{\delta^{m-1}}{\alpha_0}\\
&\quad\times 2^\frac{mn\delta^{m-1}}{\alpha_0\left(n-1\right)}
\left\|\left(u+h\right)^\frac{\alpha_0}{\delta^m}\right\|_{W^{1,p^*}\left(B_{\rho_m}\right)}^\frac{\delta^m}{\alpha_0}.
\end{align*}
Then by iteration, we obtain
\begin{align}\label{xidie}
\left\|\left(u+h\right)^{\alpha_0}\right\|^\frac{1}{\alpha_0}_{W^{1,p^*}(B_{\rho_0})}
&\leq  \left(\frac{C\Lambda(B_1)^\frac{1}{2}\alpha_0^2}{\left(\tau-\theta\right)\left(1-\frac{2\alpha_0}{\delta}\right)}\right)^\frac{\delta^m-1}{\left(\delta-1\right)\alpha_0}
\left(\frac{2^\frac{n}{n-1}}{\delta}\right)^{\frac{1}{\alpha_0}
	\left(\frac{\left(m+1\right)}{\delta-1}\delta^{m+1}-\frac{\delta^{m+1}-1}{\left(\delta-1\right)^2}\right)}\notag\\
&\quad\times\left\|\left(u+h\right)^\frac{\alpha_0}{\delta^m}\right\|^\frac{\delta^m}{\alpha_0}_{W^{1,p_*}\left(B_{\rho_m}\right)}.
\end{align}
Subsequently, we fix $m\in \mathbb{N}$ satisfying $\frac{2p\alpha_0 }{\delta^m\left(p-1\right)}\leq p_0\leq \frac{2p\alpha_0}{\delta^{m-1}\left(p-1\right)}$, and apply H\"older's inequality to \eqref{ine11} to estimate the right-hand of \eqref{xidie} as follows
\begin{align*}
\left\|\nabla \left(u+h\right)^\frac{\alpha_0}{\delta^m}\right\|_{L^{p_*}\left(B_{\rho_m}\right)}
&\leq C \left\|\nabla \left(u+h\right)^\frac{\alpha_0}{\delta^m}\right\|_{L^{\frac{2q}{q+1}}\left(B_{\frac{\tau+\theta}{2}}\right)}\\
&\leq C_5\left\|\left(u+h\right)^\frac{\alpha_0}{\delta^m}\right\|_{L^\frac{2p}{p-1}\left(B_\frac{3\tau+\theta}{4}\right)}\leq C_6\left\|u+h\right\|^\frac{\alpha_0}{\delta^m}_{L^{p_0}\left(\frac{3\tau+\theta}{4}\right)},
\end{align*}
where $C_5,C_6$ depend on $\tau,\theta,\alpha_0,\Lambda(B_1),n, p,q,s$. Further, inserting the above inequality into \eqref{xidie}, one has
\begin{align}\label{ine12}
\left\|\left(u+h\right)^{\alpha_0}\right\|^\frac{1}{\alpha_0}_{W^{1,p^*}(B_{\rho_0})}
\leq C\Lambda(B_1)^{\frac{p\delta}{\left(p-1\right)p_0\left(\delta-1\right)}-\frac{1}{2\alpha_0\left(\delta-1\right)}}\left\| u+h\right\|_{L^{p_0}\left(B_\frac{3\tau+\theta}{4}\right)}.
\end{align}

Once again, by Sobolev inequality,  \eqref{ine11} and \eqref{zuih}, we find
\begin{align}\label{ine13}
\left\|\left(u+h\right)^{\alpha_0}\right\|_{L^{q_*}\left(B_\theta\right)}
\leq C\left\|\left(u+h\right)^{\alpha_0}\right\|_{W^{1,\frac{2q}{q+1}}\left(B_\theta\right)}
\leq  C\left\|\left(u+h\right)^{\alpha_0}\right\|_{W^{1,p_*}\left(B_\frac{\tau+3\theta}{4}\right)},
\end{align}
where $C=C(\tau,\theta,n,\Lambda(B_1),q,p,\alpha_0)$.
Then a combination of \eqref{ine12} and \eqref{ine13} yields the desired	claim \eqref{5y}.

This completes the proof of Lemma \ref{THM104}.

\section{Proof of Example \ref{e1} -- \ref{e3}}
In this section we first construct Example \ref{e1} to demonstrate that $s=\frac{nq}{2q-n}$ is optimal. Our example is intuitively straightforward.

{\bf Proof of Example \ref{e1}. }
Let $\Omega$ be the ball $B=B_\frac{1}{4}(0)$. Define
\begin{align*}
{\bf{A}}(x)=
\begin{pmatrix}
a_1(x) &0&\cdots&0\\
0 & a_1(x)&\cdots&0\\
\vdots & \vdots&\ddots&\vdots\\
0&0&\cdots&a_1(x)
\end{pmatrix},~a_1(x)=|x|^\beta\left(\log\frac{1}{|x|}\right)^\theta,
\end{align*}
and
\begin{align*}
f(x)&=(n+1)|x|^{\beta-1}\left(\log\frac{1}{|x|}\right)^\theta\int_{B_\frac{1}{2}(0)}|y|^{1-n}\left(|x|+|y|\right)^{-2}\frac{1}{\log\frac{1}{|y|}}{\rm d}y\\
&\quad-\theta|x|^{\beta-1}\left(\log\frac{1}{|x|}\right)^{\theta-1}\int_{B_\frac{1}{2}(0)}|y|^{1-n}\left(|x|+|y|\right)^{-2}\frac{1}{\log\frac{1}{|y|}}{\rm d}y\\
& \quad-2|x|^\beta\left(\log\frac{1}{|x|}\right)^\theta\int_{B_\frac{1}{2}(0)} |y|^{1-n}\left(|x|+|y|\right)^{-3}\frac{1}{\log\frac{1}{|y|}}{\rm d}y,
\end{align*}
where $\beta=\frac{n}{q}<2$, $\theta=\frac{1}{2}+\frac{1}{q}-\frac{1}{n}$.

It is easy to observe that
\begin{align*}
u(x)=\int_{B_\frac{1}{2}(0)}|y|^{1-n}\left(|x|+|y|\right)^{-1}\frac{1}{\log\frac{1}{|y|}}{\rm d}y
\end{align*}
is the weak solution of \eqref{Equ01} and $u(0)=+\infty$. To make the argument rigorous, we exploit some technique from [\ref{CUR}] to show that
\begin{align*}
\int_{\Omega} {\bf{A}}(x)\nabla u\nabla \varphi {\rm d}x=\int_{\Omega} f\varphi {\rm d}x, ~~\forall \varphi\geq0\in C^\infty_0(\Omega).
\end{align*}
For each $k\geq 2$, let $\chi_k$ be a $C^\infty$, non-negative and radial function satisfying
\begin{align*}
\chi_k(x) =
\begin{cases}
0, &|x|\leq 2^{-k-2},\\
1, &2^{-k-1}\leq |x|<\frac{1}{4}.
\end{cases}
\end{align*}
Define $f_k=\chi_k f$. Notice that each $f_k$ is continuous, and $$u_k=\chi_k\int_{B_\frac{1}{2}(0)}|y|^{1-n}\left(|x|+|y|\right)^{-1}\frac{1}{\log\frac{1}{|y|}}{\rm d}y\in C^2(\Omega)$$ 
is the solution of $-\dive\left({\bf{A}}(x)\nabla u_k\right)=f_k$, in the weak sense,
\begin{align*}
\int_{\Omega}{\bf{A}}(x)\nabla u_k\nabla \varphi {\rm d}x=\int_{\Omega}f_k \varphi {\rm d}x,~~\forall \varphi \geq0\in C^\infty_0(\Omega).
\end{align*}
The strong convergence $f_k\rightarrow f$ in $L^s(\Omega)$ implies
$$\int_{\Omega}{\bf{A}}(x)\nabla u_k\nabla \varphi {\rm d}x\rightarrow \int_{\Omega}{\bf{A}}(x)\nabla u\nabla\varphi {\rm d}x,~~k\rightarrow \infty.$$

Next, we will verify that $\lambda^{-1}(x)\in L^q(\Omega)$.
Indeed, by using a polar coordinate transformation, we obtain
\begin{align*}
\int_{B_\frac{1}{2}(0)}\left|\frac{1}{\lambda(x)}\right|^q{\rm d}x
=&\;\alpha_n\int_{0}^\frac{1}{2}\left|\frac{1}{r^\beta\left(\log\frac{1}{r}\right)^\theta}\right|^q r^{n-1} {\rm d}r
=\alpha_n\int_{0}^{\frac{1}{2}}\frac{1}{\left(\log\frac{1}{r}\right)^{\theta q}}r^{n-1-\beta q} {\rm d}r\\
=&\;\alpha_n\int_{2}^\infty\frac{1}{\left(\log t\right)^{\theta q}}\left(\frac{1}{t}\right)^{n+1-\beta q} {\rm d}t
=\alpha_n\int_{2}^{\infty}\frac{1}{\left(\log t\right)^{\theta q}}\frac{1}{t} {\rm d}t
<+\infty,
\end{align*}
where $\alpha_n$ is the surface area of an $n$-dimensional sphere.

To complete the proof of Example \ref{e1}, we will demonstrate that $f\in L^\frac{nq}{2q-n}(\Omega)$ when $n\geq 3$. It is sufficient to demonstrate one term of $f$ that belongs to $L^\frac{nq}{2q-n}(\Omega)$, the remaining terms can be handled in a similar manner.
For example, we will show that
\begin{align*}
|x|^\beta\left(\log\frac{1}{|x|}\right)^\theta\int_{B_\frac{1}{2}(0)}|y|^{1-n}\left(|x|+|y|\right)^{-3}\frac{1}{\log\frac{1}{|y|}}{\rm d}y \in L^\frac{nq}{2q-n}\left(B_\frac{1}{4}(0)\right).
\end{align*}
In fact, we know
\begin{align}\label{1t}
&\int_{B_\frac{1}{4}(0)}\left||x|^\beta\left(\log\frac{1}{|x|}\right)^\theta\int_{B_\frac{1}{2}(0)}|y|^{1-n}\left(|x|+|y|\right)^{-3}\frac{1}{\log\frac{1}{|y|}}{\rm d}y\right|^\frac{nq}{2q-n} {\rm d}x\notag\\
=&\int_{B_\frac{1}{4}(0)}|x|^\frac{\beta nq}{2q-n}\left(\log\frac{1}{|x|}\right)^\frac{\theta nq}{2q-n}\left|\int_{B_\frac{1}{2}(0)}|y|^{1-n}\left(|x|+|y|\right)^{-3}\frac{1}{\log\frac{1}{|y|}}{\rm d}y\right|^\frac{nq}{2q-n} {\rm d}x.
\end{align}
Then, we will estimate
\begin{align}\label{in1}
\int_{B_\frac{1}{2}(0)}|y|^{1-n}\left(|x|+|y|\right)^{-3}\frac{1}{\log\frac{1}{|y|}}{\rm d}y.
\end{align}
We decompose \eqref{in1} into two terms
\begin{align*}
\int_{B_\frac{1}{2}(0)}|y|^{1-n}\left(|x|+|y|\right)^{-3}\frac{1}{\log\frac{1}{|y|}}{\rm d}y
&=\int_{B_\frac{1}{2}(0)\cap\{|y|<|x|^\frac{1}{2}\}}|y|^{1-n}\left(|x|+|y|\right)^{-3}\frac{1}{\log\frac{1}{|y|}}{\rm d}y\\
&\quad+\int_{B_\frac{1}{2}(0)\cap\{|y|\geq |x|^\frac{1}{2}\}}|y|^{1-n}\left(|x|+|y|\right)^{-3}\frac{1}{\log\frac{1}{|y|}}{\rm d}y\\
&=I_1+I_2.
\end{align*}
By using a polar coordinate transformation, we find  $I_1$ may be estimated as follows
\begin{align}\label{I1}
I_1=&\int_{B_\frac{1}{2}(0)\cap\{|y|<|x|^\frac{1}{2}\}}|y|^{1-n}\left(|x|+|y|\right)^{-3}\frac{1}{\log\frac{1}{|y|}}{\rm d}y
\leq\alpha_n\int_{0}^{|x|^\frac{1}{2}}\left(|x|+r\right)^{-3}\frac{1}{\log\frac{1}{r}}{\rm d}r\notag\\
\leq& \;\frac{2\alpha_n}{ \log\frac{1}{|x|}}\int_{0}^{|x|^\frac{1}{2}}\left(|x|+r\right)^{-3}{\rm d}r
=\frac{\alpha_n}{ \log\frac{1}{|x|}}\left(|x|^{-2}-\left(|x|^\frac{1}{2}+|x|\right)^{-2}\right)
\leq\frac{\alpha_n}{ \log\frac{1}{|x|}}|x|^{-2}.
\end{align}
Similarly, for $I_2$, we also estimate 
\begin{align}\label{I2}
I_2\leq&\;\alpha_n\int_{|x|^\frac{1}{2}}^\frac{1}{2}\left(|x|+r\right)^{-3}\frac{1}{\log\frac{1}{r}}{\rm d}r
\leq\frac{\alpha_n}{\log 2}\left(|x|+|x|^\frac{1}{2}\right)^{-3}\int_{|x|^\frac{1}{2}}^\frac{1}{2}{\rm d}r\notag\\
=& \;\frac{\alpha_n}{\log 2}\left(|x|+|x|^\frac{1}{2}\right)^{-3}\left(\frac{1}{2}-|x|^\frac{1}{2}\right)
\leq\frac{\alpha_n}{\log 2}|x|^{-\frac{3}{2}}\left(\frac{1}{2}-|x|^\frac{1}{2}\right)
=\frac{\alpha_n}{2\log 2}|x|^{-\frac{3}{2}}.
\end{align}
By using \eqref{in1}-\eqref{I2}, \eqref{1t} has the following estimate
\begin{align*}
&\int_{B_\frac{1}{4}(0)}|x|^\frac{\beta nq}{2q-n}\left(\log\frac{1}{|x|}\right)^\frac{\theta nq}{2q-n}\left|\int_{B_\frac{1}{2}(0)}|y|^{1-n}\left(|x|+|y|\right)^{-3}\frac{1}{\log\frac{1}{|y|}}{\rm d}y\right|^\frac{nq}{2q-n} {\rm d}x\\
\leq& \;c\int_{B_\frac{1}{4}(0)}|x|^\frac{\beta nq}{2q-n}\left(\log\frac{1}{|x|}\right)^\frac{\theta nq}{2q-n} \left(\frac{1}{ \log\frac{1}{|x|}}|x|^{-2}\right)^\frac{nq}{2q-n}{\rm d}x\\
& +c\int_{B_\frac{1}{4}(0)}|x|^\frac{\beta nq}{2q-n}\left(\log\frac{1}{|x|}\right)^\frac{\theta nq}{2q-n} \left(|x|^{-\frac{3}{2}}\right)^\frac{nq}{2q-n}{\rm d}x\\
=&\;c\int_{0}^\frac{1}{4}r^{\frac{(\beta-2) nq}{2q-n}+n-1}\left(\log\frac{1}{r}\right)^{\frac{nq(\theta-1)}{2q-n}} {\rm d}r
+c\int_{0}^\frac{1}{4}r^{\frac{(\beta-\frac{3}{2}) nq}{2q-n}+n-1}\left(\log\frac{1}{r}\right)^{\frac{nq\theta}{2q-n}} {\rm d}r\\
=&\;I_3+I_4,
\end{align*}
where $c$ only depends on $n,~q$. Now, we estimate $I_3$. Notice
\begin{align*}
&\frac{(\beta-2) nq}{2q-n}+n-1=\frac{(\frac{n}{q}-2) nq}{2q-n}+n-1=-1,\\
&\frac{nq(\theta-1)}{2q-n}=\frac{n-q-2nq}{2(2q-n)}<-1,
\end{align*}
then, a simple computation shows $I_3<+\infty.$

Similarly, due to
\begin{align*}
\frac{(\beta-\frac{3}{2}) nq}{2q-n}+n-1=\frac{nq}{q-\frac{n}{2}}-1>-1,
\end{align*}
it is not hard to prove that
\begin{align*}
I_4<+\infty.
\end{align*}
Thus we arrive at
\begin{align*}
\int_{B_\frac{1}{4}(0)}\left||x|^\beta\left(\log\frac{1}{|x|}\right)^\theta\int_{B_\frac{1}{2}(0)}|y|^{1-n}\left(|x|+|y|\right)^{-3}\frac{1}{\log\frac{1}{|y|}}{\rm d}y\right|^\frac{nq}{2q-n} {\rm d}x<+\infty.
\end{align*}
For the case $n=2$, by a similar argument, we could  also verify $f\in L^s(B_\frac{1}{4}(0))$, for $s\in[1,\frac{q}{q-1})$.

{\bf Proof of Example \ref{e2}. } We take $\beta=2$ in the proof of Example \ref{e1}, and we have
\begin{align*}
a_1(x)=|x|^2\left(\log\frac{1}{|x|}\right)^\theta,
\end{align*}
\begin{align*}
f(x)&=(n+1)|x|\left(\log\frac{1}{|x|}\right)^\theta\int_{B_\frac{1}{2}(0)}|y|^{1-n}\left(|x|+|y|\right)^{-2}\frac{1}{\log\frac{1}{|y|}}{\rm d}y\\
&\quad+\theta|x|\left(\log\frac{1}{|x|}\right)^{\theta-1}\int_{B_\frac{1}{2}(0)}|y|^{1-n}\left(|x|+|y|\right)^{-2}\frac{1}{\log\frac{1}{|y|}}{\rm d}y\\
&\quad-2|x|^2\left(\log\frac{1}{|x|}\right)^\theta\int_{B_\frac{1}{2}(0)} |y|^{1-n}\left(|x|+|y|\right)^{-3}\frac{1}{\log\frac{1}{|y|}}{\rm d}y,
\end{align*}
where $\theta=\frac{5}{2n}$.
We observe
\begin{align*}
u(x)=\int_{B_\frac{1}{2}(0)} |y|^{1-n} \left(|x|+|y|\right)^{-1}\frac{1}{\log\frac{1}{|y|}}{\rm d}y
\end{align*}
is the weak solution of \eqref{Equ01}, and $u(0)=\infty$. The rest of procedure is similar to Example \ref{e1}. Finally, we find $f\in L^\infty(B_\frac{1}{4}(0))$ when $n\geq 3$ and $f\in L^s(B_\frac{1}{4}(0))$, $s\in [1,+\infty)$ when $n=2$.

{\bf Proof of Example \ref{e3}. } We set $\beta=\frac{n}{q}\in \left(2, \frac{2n}{n-1}\right),~\theta=1$ in Example \ref{e1}. Let
\begin{align*}
a_1(x)=|x|^\frac{n}{q}\left(\log \frac{1}{|x|}\right)^\frac{2}{q},
\end{align*}
\begin{align*}
f(x)&=(n+1)|x|^{\frac{n}{q}-1}\left(\log\frac{1}{|x|}\right)\int_{B_\frac{1}{2}(0)}|y|^{1-n}\left(|x|+|y|\right)^{-2}\frac{1}{\log\frac{1}{|y|}}{\rm d}y\\
&\quad-|x|^{\frac{n}{q}-1}\int_{B_\frac{1}{2}(0)}|y|^{1-n}\left(|x|+|y|\right)^{-2}\frac{1}{\log\frac{1}{|y|}}{\rm d}y\\
&\quad-2|x|^\frac{n}{q}\left(\log\frac{1}{|x|}\right)\int_{B_\frac{1}{2}(0)} |y|^{1-n}\left(|x|+|y|\right)^{-3}\frac{1}{\log\frac{1}{|y|}}{\rm d}y,
\end{align*}
and $f\in L^\infty(B_\frac{1}{4}(0))$. Then we take $$u(x)=\int_{B_\frac{1}{2}(0)} |y|^{1-n} \left(|x|+|y|\right)^{-1}\frac{1}{\log\frac{1}{|y|}}{\rm d}y,$$ by a similar argument as above, we could verify $u$ is the weak solution of \eqref{Equ01}.
\begin{remark}\label{bor}
	\rm It's interesting to see that Example \ref{e3} could be extended directly to the case of $\frac{1}{p}+\frac{1}{q}=\frac{2}{n-1},~n\geq3$. Indeed, when we take $q=\frac{2}{n-1},~p=+\infty$, we find that there exists $f\in L^\infty(\Omega)$, which makes the weak solution $u$ unbounded.
\end{remark}

%\noindent{\bf Acknowledgments.} The authors express their thanks to
%the referees for their helpful comments and suggestions.

\end{document}